\documentclass[oneside, a4paper]{amsart}

\usepackage{amssymb}

\usepackage{amsmath}

\usepackage{amscd,amsopn}

\usepackage{graphicx}
\usepackage{amsmath,  amsfonts,  amssymb,  amsthm}
\usepackage{xypic}
\usepackage{verbatim}

\usepackage{tikz-cd}

\usepackage{mathrsfs}

\usepackage{stmaryrd}

\usepackage{enumitem}

\usepackage[all]{xy}

\usepackage{textcomp}

\usepackage{amsbsy}

\usepackage{xcolor}

\usepackage{etex}
\usepackage{hyperref}
\hypersetup{
  colorlinks = true,
  linkcolor  = black
}

\usepackage[alphabetic]{amsrefs}

\AtBeginDocument{%
   \def\MR#1{}
}

\usepackage{tikz}

\author{Yi Xie}
\email{yixie@pku.edu.cn}
\address{Beijing International Center for Mathematical Research, Peking University, Beijing 100871, China}

\DeclareMathOperator{\dimm}{dim}

\DeclareMathOperator{\rankk}{rank}

\DeclareMathOperator{\pt}{pt}
\DeclareMathOperator{\II}{I}
\DeclareMathOperator{\THI}{THI}
\DeclareMathOperator{\AHI}{AHI}
\DeclareMathOperator{\AKh}{AKh}
\DeclareMathOperator{\TKh}{TKh}
\DeclareMathOperator{\muu}{\mu^{orb}}
\DeclareMathOperator{\idd}{id}
\DeclareMathOperator{\Kh}{Kh}
\DeclareMathOperator{\Khr}{Khr}
\DeclareMathOperator{\RI}{RI}
\DeclareMathOperator{\THIo}{THI^{odd}}
\DeclareMathOperator{\SHI}{SHI}

\newtheorem{THE}{Theorem}[section]

\newtheorem{thm-defn}[THE]{Theorem/Definition}
\newtheorem{LE}[THE]{Lemma}
\newtheorem{PR}[THE]{Proposition}
\newtheorem{COR}[THE]{Corollary}

\theoremstyle{definition}
\newtheorem{DEF}[THE]{Definition}

\newtheorem{eg}[THE]{Example}

\theoremstyle{remark}
\newtheorem*{rmk}{Remark}

\numberwithin{equation}{section}

\title{Instantons and Annular Khovanov Homology }
\date{}

\begin{document}
\maketitle


\begin{abstract}
In this paper, we introduce the annular instanton Floer homology which is defined for links in a thickened annulus. It is
an analogue of the annular Khovanov homology.
A spectral sequence whose second page is the annular Khovanov homology and which converges to the annular 
instanton Floer homology is constructed. As an application of this spectral sequence, we prove that the annular Khovanov homology detects
the unlink in the thickened annulus (assuming all the components are null-homologous). Another application is a new proof
of Grigsby and Ni's result that tangle Khovanov homology distinguishes braids from other tangles. 
\end{abstract}

\section{Introduction}

In \cite{Kh-Jones}, Khovanov defined a bi-graded homology group for a link in $S^3$ which categorifies the Jones polynomial of the link.
Since then, the relationship between Khovanov homology and Floer homology theories in different settings has been studied a lot.

The first such result is due to Ozsv\'{a}th and Szab\'{o} \cite{OS-ss}: they constructed a spectral sequence whose
$E_2$-page is the (reduced) Khovanov homology of the mirror image of a link $L$ in $S^3$ and which converges to the (hat version) Heegaard Floer 
homology of the branched double cover of $S^3$ over $L$.

Similar results were proved by Bloom \cite{Bloom} 
in the monopole Floer homology case and Scaduto \cite{Sca} in the framed intanton homology case.
All the above spectral
sequences 
are related to the Floer homology of the branched double cover of $S^3$. In \cite{KM:Kh-unknot}, Kronheimer and
Mrowka constructed a spectral sequence converging to a version of singular instanton Floer homology of a link. 
As applications of their
spectral sequence, it was shown that Khovanov homology detects the unknot \cite{KM:Kh-unknot} and 
 the trefoil \cite{BS}.

In \cite{APS}, Khovanov homology was generalized for (framed) links in 
a $I$-bundle over a compact surface (possibly with boundary).
In particular, their definition works for links in a thickened annulus. 
We call it the \emph{annular Khovanov homology} in this situation. In \cite{Rob}, Roberts 
constructed a spectral sequence whose $E_2$-page is 
the annular Khovanov homology of the mirror of a link and which converges to the (hat) Heegaard Floer homology of  
 certain branched double cover. 
In this paper, we define an analogue of the annular Khovanov homology in the singular instanton Floer homology setting, 
called the \emph{annular instanton Floer homology}. The annular instanton Floer homology is related to the annular Khovanov homology by
a spectral sequence. 
\begin{THE}\label{AKh-AHI*}
Let $L$ be a link in the thickened annulus and $\overline{L}$ be its mirror image.
There is a spectral sequence whose $E_2$-page is the annular Khovanov homology
$\AKh(\overline{L};\mathbb{C})$ and which converges to the annular instanton Floer homology $\AHI(L;\mathbb{C})$.
\end{THE}
The annular Khovanov homology $\AKh(\overline{L})$ is  triply graded where the first two gradings are similar to the
homological grading and the quantum grading in the standard Khovanov homology. The third grading is new and a similar grading is
 defined for $\AHI(L;\mathbb{C})$. We also prove that the above spectral sequence preserves this grading (Theorem \ref{AKh-AHI}).  

Suppose $T$ is a balanced admissible tangle (see Definition \ref{tangle-def}) 
in $I\times D$ where $I=[-1,1]$ and $D$ is the standard 2-disk. Closing up $T$,  we obtain 
a link $\hat{T}$ in $S^1\times D$ which is the same as the thickened annulus. 
Let $i$ be the number of end points of $T$ on the top disk.
The degree $i$ summand $\AHI(\hat{T},i;\mathbb{C})$ of 
$\AHI(L;\mathbb{C})$ is just 
the instanton Floer homology $\THI(T)$ for tangles introduced in \cite{Street}.  On the other hand,
the degree $i$ summand $\AKh(\hat{T},i)$ of $\AKh(\hat{T})$ is isomorphic to the Khovanov homology $\TKh(T)$ for tangles
introduced in \cite{Kh-color} (see also \cite{GW-annular} and \cite{GN-suture}). 
Therefore degree $i$ summand of the spectral sequence of Theorem \ref{AKh-AHI*} gives the following.
\begin{COR}\label{TKh-THI*}
Let $T\subset I\times D$ be a balanced admissible tangle and $\overline{T}$ be its mirror image.
There is a spectral sequence whose $E_2$-page is the tangle Khovanov homology
$\TKh(\overline{T};\mathbb{C})$ and which converges to the tangle instanton Floer homology $\THI(T)$.
\end{COR}
This corollary is an analogue of the spectral sequence relating tangle Khovanov homology and sutured Heegaard Floer homology
in \cite{GW-annular}. By a theorem of Street \cite{Street}, 
$\THI(T)\cong \mathbb{C}$ if and only if $T$ is isotopic to a braid. Therefore we obtain a new proof of the following.
\begin{THE}\label{braid-detection-TKh*}
Let $T\subset I\times D$ be a balanced admissible tangle. Then $\TKh(T;\mathbb{C})\cong\mathbb{C}$ if and only if $T$ is isotopic to a braid.
\end{THE} 
This theorem (with $\mathbb{Z}/2$ coefficients) is first proved in \cite{GN-suture}.
The original proof uses the spectral sequence in \cite{GW-annular} as well as the unknot detection result in \cite{KM:Kh-unknot}.

Khovanov also defined a sequence of invariants $\Khr_n(K)$ of a knot $K\subset S^3$  
which categorify the (reduced) $n$-colored Jones polynomials in \cite{Kh-color}. 
In \cite{GW-color},
they generalized Ozsv\'{a}th-Szab\'{o}'s spectral sequence to the case of colored Khovanov homology. Similar to their result, another
application of Corollary \ref{TKh-THI*} is a generalization of the spectral sequence in \cite{KM:Kh-unknot}.
\begin{THE}\label{colored-ss}
Let $K$ be a knot in $S^3$ and $\overline{K}$ be its mirror image.
There is a spectral sequence whose $E_2$-page is the reduced $n$-colored Khovanov homology
$\Khr_n(\overline{K};\mathbb{C})$ and which converges to  the reduced singular instanton Floer homology $\II^\natural(K;\mathbb{C})$.         %
\end{THE}

Let $L\subset S^1\times D$ be a link with all the components null-homologous. A properly embedded surface $\Sigma$ in $S^1\times D$ 
is called an \emph{admissible surface} for $L$ 
if it is  a connected orientable surface which is bounded by a non-null-homologous circle in $S^1\times \partial D$  
and disjoint from $L$. This is the same as saying $\partial \Sigma\cong S^1$, $\Sigma\cap L=\emptyset$ and $[\Sigma,\partial \Sigma]$
generates $H_2(S^1\times D, \partial(S^1\times D);\mathbb{Z})\cong \mathbb{Z}$.
With the help of Kronheimer and Mrowka's instanton Floer homology for sutured manifolds \cite{KM:suture}, we prove a non-vanishing result 
for the annular instanton Floer homology.
\begin{THE}\label{AHI-non-vanish*}
Suppose $L\subset S^1\times D$ is a link with all the components null-homologous and $\Sigma$ is an
admissible surface for $L$ with minimal genus, then we have
\begin{equation*}
\AHI(L,\pm 2g(\Sigma)) \neq 0
\end{equation*} 
\end{THE}
If $\AHI(L;\mathbb{C})$ is supported in degree $0$, then by the above theorem there is a properly embedded disk 
in $S^1\times D$ whose boundary circle is homologous to $\partial D$ in $S^1 \times \partial D$ and 
which is disjoint from $L$. This means $L$ is included in a three-ball $B^3\subset S^1\times D$. 
By Theorem \ref{AKh-AHI*} (and the discussion below it), we have the following.
\begin{COR}
Suppose $L\subset S^1\times D$ is a link with all the components null-homologous and 
$\AKh(L)$ is supported in degree $0$ (the third grading), then $L$ is included in a three-ball $B^3\subset S^1\times D$.
\end{COR} 
When an annulus link $L$ is included in a three-ball $B^3\subset S^1\times D$,
$\AKh(L;0)$ ($0$ is the third grading) is 
isomorphic to $\Kh (L)$ (view $L$ as a link in $B^3$ to define $\Kh(L)$)
as a bi-graded abelian group. 
Together with the result that Khovanov homology detects the unlink \cite{Kh-unlink,HN-module}, we have the following.
\begin{COR}
Suppose $L\subset S^1\times D$ is a link with $k$ components and all the components are null-homologous. If
 $\AKh(L)$ is supported in degree $0$ (the third grading) and $\AKh(L,0)\cong \Kh(U_k)$  where $U_k$ denotes the 
 unlink with $k$ components, then $L$ is an unlink in $S^1\times D$.
\end{COR}
An unlink in $S^1\times D$ means a link bounding a collection of disjoint disks in $S^1\times D$. For any annulus link $L$,
there is a spectral sequence whose $E_1$-page is $\AKh(L)$ and which converges to $\Kh(L)$ where 
we use the standard embedding $S^1\times D\subset S^3$ to make $L$ into a link in $S^3$ and define $\Kh(L)$
\cite{Rob}. In particular, 
 if $K$ is a knot with $\rankk \AKh(K)= 2$, then we have $\rankk \Kh(K)\le 2$ hence $K$ is an unknot in $S^3$ by \cite{KM:Kh-unknot}.
 However, this argument cannot tell us that $K$ is an unknot in $S^1\times D$. For example, the knot in Figure \ref{Whitehead-Knot}
 is an unknot in $S^3$ but it is a null-homologous non-trivial knot in $S^1\times D$. 

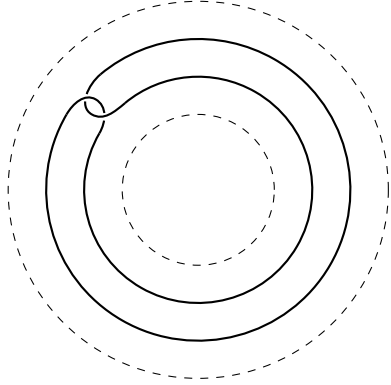
\begin{figure}
\centering
\begin{tikzpicture}

\draw[dashed] (0,0) circle [radius=1];  \draw[dashed] (0,0) circle [radius=2.5];
\draw[thick]  (0-1.3,0.75) arc [radius=1.5, start angle=150, end angle=490];
\draw[thick]  (0-1.732,1) arc [radius=2, start angle=150, end angle=490]; 

\draw[thick, dash pattern=on 0.18cm off 0.1cm on 100cm] (0-1.3,0.75) to [out=60, in=-30] (-1.35, 1.2) to [out=150, in=60]  (0-1.732,1); 
\draw[thick, dash pattern=on 0.7cm off 0.1cm on 100cm]  (0-0.964 ,1.149) to [out=220, in=-30] (-1.4, 1.0) to [out=150, in=220]  (0-1.285, 1.532);
\end{tikzpicture}
\caption{A null-homologous knot in the thickened annulus}\label{Whitehead-Knot}
\end{figure}

This paper is organized as follows:
Section \ref{AKh} is a review of the definitions of annular Khovanov homology and tangle Khovanov homology. 
In Section \ref{tangle},  we present necessary background on Street's instanton Floer
 homology theory for tangles which will be used later. We define the annular instanton Floer homology 
 and prove Theorem \ref{AHI-non-vanish*}
 in Section \ref{AHI}.
 In Section \ref{SS}, we prove
 Theorem \ref{AKh-AHI*}. Finally we apply Theorem \ref{AKh-AHI*} to derive  Corollary \ref{TKh-THI*}, Theorem \ref{braid-detection-TKh*}
 and Theorem \ref{colored-ss} in Section \ref{app-ss}.

In this paper all versions of Khovanov homologies are defined over $\mathbb{Z}$ unless otherwise specified. All the instanton Floer homologies
are defined over $\mathbb{Z}$ unless we specify the coefficients or the Floer homology group is defined as a generalized eigenspace
for an operator. We use complex coefficients whenever we take the generalized eigenspaces.  

\emph{Acknowledgments.} 
 The author would like to thank the hospitality of Yi Ni, Siqi He and
Qiongling Li during a visit to CalTech where most progress of this work was made. This paper owes a lot to the work of 
Kronheimer, Mrowka and Street. The  author was supported by National Key R\&D Program of China 2020YFA0712801 and NSFC 12071005. 

\section{Annular Khovanov Homology}\label{AKh}
In this section we will review the {annular Khovanov homology} defined for (framed) links in an thickened annulus, 
which is a special case of 
the invariants defined in \cite{APS}.
An elaboration and adjustment of their definition with $\mathbb{Z}/2$-coefficients
 can be found in \cite{Rob}. Since we need to work in characteristic $0$, we will follow the discussion in \cite{Rob} to
review the definition and deal with the sign carefully in this section. 
The annular Khovanov homology of a link is a triply-graded abelian group. But
we will omit the discussion of the first two gradings because only the third grading is important for us in this paper.

We use $A$ to denote an annulus.
Let $L\subset A\times I$ be a link with a projection to $A\cong A\times \{0\}$. The projection gives a diagram $D$ with $c$ crossings. Fix an order for the
crossings.
Given any $v=(v_1,v_2,\cdots, v_c)\in \{0,1\}^c$, we obtain
a collection of circles in $A$ by resolving the crossings using 0-smoothing or 1-smoothing determined by $v$ (see Figure \ref{01smoothing}).
\begin{figure}
\centering
\begin{tikzpicture}
\draw[thick] (1,-1) to (-1,1); \draw[thick,dash pattern=on 1.3cm off 0.25cm] (1,1) to (-1,-1);  \node[below] at (0,-1.2) {A crossing};

\draw[thick] (2,1)  to [out=315,in=180]  (3,0.3) to [out=0,in=225]   (4,1);
\draw[thick] (2,-1)  to [out=45,in=180]  (3,-0.3) to [out=0,in=135]   (4,-1);  \node[below] at (3,-1.2) {0-smoothing};

\draw[thick] (5,1)  to [out=315,in=90]  (5.7,0) to [out=270,in=45]    (5,-1);  \node[below] at (6,-1.2) {1-smoothing};
\draw[thick] (7,1)  to [out=225,in=90]  (6.3,0) to [out=270,in=135]   (7,-1);
\end{tikzpicture}
\caption{Two types of smoothings}\label{01smoothing}
\end{figure}
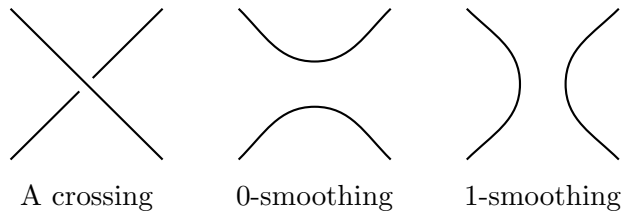
There are two types of circles: circles that bound disks and homologically non-trivial circles. We call them trivial circle and non-trivial circles respectively.
Define graded free abelian groups $V:=\mathbb{Z}\{\mathbf{v}_+,\mathbf{v}_-\}$ with $\deg \mathbf{v}_\pm=0$ and
$W:=\mathbb{Z}\{\mathbf{w}_+,\mathbf{w}_-\}$ with $\deg \mathbf{w}_\pm=\pm 1$. 
Given any $v\in \{0,1\}^c$, suppose there are $a$ trivial circles and $b$ non-trivial circles
in the resolution determined by $v$, then we define
\begin{equation*}
  CKh_v(L):= V^{\otimes a} \otimes W^{\otimes b}
\end{equation*}
This means that we assign a copy of $V$ to each trivial circle and a copy of $W$ to each non-trivial circle.
If we change the 0-smoothing at a crossing into a 1-smoothing, then in the resolution either two circles merge into one circle or one circle splits into two circles.
We use $u\in \{0,1\}^c$ to record the new resolution. To define a map from $CKh_v$ to $CKh_u$, it is enough to specify maps between abelian groups assigned to the
circles involved in the merging or splitting. In the case two trivial circles merge into one trivial circle, the maps are
\begin{align*}
  \mathbf{v}_+ \otimes \mathbf{v}_+ &\mapsto \mathbf{v}_+,  &\mathbf{v}_+& \otimes \mathbf{v}_- \mapsto \mathbf{v}_- \\
  \mathbf{v}_- \otimes \mathbf{v}_+ &\mapsto \mathbf{v}_-,  &\mathbf{v}_-& \otimes \mathbf{v}_- \mapsto 0 \\
\end{align*}
In the case a trivial circle and a non-trivial circle merge into a non-trivial circle, the maps are
\begin{align*}
  \mathbf{v}_+\otimes \mathbf{w}_+ &\mapsto \mathbf{w}_+, &\mathbf{v}_+&\otimes \mathbf{w}_- \mapsto \mathbf{w}_- \\
  \mathbf{v}_-\otimes \mathbf{w}_+ &\mapsto 0,           &\mathbf{v}_-&\otimes \mathbf{w}_- \mapsto 0\\
\end{align*}
In the case two non-trivial circles merge into a trivial circle, the maps are
\begin{align*}
  \mathbf{w}_+\otimes \mathbf{w}_+ &\mapsto 0,   &\mathbf{w}_-&\otimes \mathbf{w}_- \mapsto 0\\
  \mathbf{w}_+\otimes \mathbf{w}_- &\mapsto \mathbf{v}_-, &\mathbf{w}_-& \otimes \mathbf{w}_+ \mapsto \mathbf{v}_-\\
\end{align*}
If one circle splits into two circles, there are also cases obtained by reversing the above three cases. The relevant maps are
\begin{align*}
   \mathbf{v}_+ &\mapsto \mathbf{v}_+\otimes \mathbf{v}_- + \mathbf{v}_-\otimes \mathbf{v}_+,  &\mathbf{v}_-& \mapsto \mathbf{v}_-\otimes \mathbf{v}_-; \\
   \mathbf{w}_+ &\mapsto \mathbf{v}_-\otimes \mathbf{w}_+, &\mathbf{w}_-& \mapsto \mathbf{v}_-\otimes \mathbf{w}_-; \\
  \mathbf{v}_+ &\mapsto \mathbf{w}_+\otimes \mathbf{w}_- + \mathbf{w}_-\otimes \mathbf{w}_+ , &\mathbf{v}_-& \mapsto 0.
\end{align*}
In this way, we obtain a map $d_{vu}: C_v \to C_u$ when $u-v=e_i$ where $e_i$ denotes the $i$-th standard basis vector of $\mathbb{R}^c$.
This map preserves the grading. Now we define
\begin{equation*}
  CKh(L):=\bigoplus_{v\in \{0,1\}} CKh_v(L)
\end{equation*}
and equip it with a differential
\begin{equation*}
  D:= \sum_i \sum_{u-v=e_i}  (-1)^{\sum_{i<j\le c}v_j}  d_{vu}
\end{equation*}
Since $D$ preserves the grading, the homology
\begin{equation*}
  \AKh(L):=H_\ast(CKh(L),D)
\end{equation*}
is a graded abelian group. We denote the degree $m$ summand by $\AKh(L,m)$.
\begin{THE}[{{\cite{APS}}}]
The annular Khovanov homology $\AKh(L)$ is a well-defined link invariant: it does not depend on the choice of diagram $D$ or the order of the crossings.
\end{THE}
If $L$ is included in a three-ball $B^3\subset A\times I$, then 
for some diagram 
there is no non-trivial circle in any resolution. In this case,
we have $\AKh(L)$ is supported at degree $0$ and $\AKh(L)\cong \Kh(L)$. This isomorphism respects the homological and quantum gradings
on both sides.

We use $D$ to denote a 2-dimensional disk and $D^\pm$ to denote $ \{\pm 1\}\times D\subset I\times D$ where $I=[-1,1]$.
\begin{DEF}\label{tangle-def}
A \emph{tangle} in $I\times D$ is a properly embedded compact 1-manifold
with (possibly empty) boundary.
A tangle $T\subset I\times D$ is called \emph{admissible} if $T\cap (I\times \partial D)  =\emptyset $ and $\partial T \subset  \partial I\times D$.
$T$ is called \emph{balanced} if $|T\cap D^+|=|T\cap D^-|$.
$T$ is called \emph{vertical} if $T$ has no closed component and every strand of $T$ has one end point in $D^+$ and another end point
in $D^-$.
\end{DEF}

Khovanov homology for tangles (we will call it \emph{tangle Khovanov homology} and denote it by $\TKh$)
is introduced by Khovanov in \cite{Kh-color}. An elaboration of Khovanov's construction in $\mathbb{Z}/2$-coefficients
can be found in \cite{GW-color} and \cite{GN-suture}. It can also be defined as a direct summand of the annular Khovanov homology for a specific link.
Let $T\subset I\times D$ be a balanced admissible tangle such that $|T\cap D^+|=m$. We can close up
$T$ to obtain a link $\hat{T}\subset S^1\times D$. Notice that $\hat{T}$ is not unique.
 By \cite{GW-annular}*{Theorem 3.1} or \cite{GN-suture}*{Proposition 2.4}, we have
\begin{equation}\label{TKh=AKh}
  \TKh(T)\cong \AKh(\hat{T},m)
\end{equation}
In particular, $\AKh(\hat{T},m)$ does not depend on the way we close up $T$.

\section{Instanton Floer Homology for Tangles}\label{tangle}
\subsection{Singular instanton Floer homology and excision}
 The theory of singular instanton Floer homology is developed by Kronheimer and Mrowka in
\cite{KM:YAFT, KM:Kh-unknot}. We will follow the version used in \cite{KM:Kh-unknot}. Recall that given a triple $(Y,K,w)$ where
\begin{itemize}
  \item $Y$ is a closed connected oriented three-manifold,
  \item $K$ is a link in $Y$,
  \item $\omega$ is an embedded 1-manifold in $Y$ meeting $K$ normally at $\partial w$.
\end{itemize}
Then $\omega$ determines an orbifold $SO(3)$ bundle $\check{P}\to \check{Y}$ (more precisely, the \emph{singular bundle data} 
in \cite{KM:Kh-unknot}) on $Y$ with orbifold points $K$ and
$w_2(\check{P})$ the Poincar\'{e} dual of $[\omega]$. 
We say the triple $(Y, K,\omega)$ is \emph{admissible} if there is an embedded surface $\Sigma\subset Y$ such that 
either
\begin{itemize}
  \item $\Sigma$ is disjoint from $K$ and $\omega\cdot \Sigma$ is odd; or
  \item  $\Sigma$ intersects $K$ transversely and $\Sigma\cdot K$ is odd.
\end{itemize}
In this situation, the instanton Floer homology group
$$
\II(Y,K,\omega)
$$
is defined in \cite{KM:Kh-unknot} as the Morse homology of the Chern-Simons functional on the space of orbifold connections with the
asymptotic holonomy around $K$ to be an order 2 element in $SO(3)$.
It is a relatively $\mathbb{Z}/4$-graded abelian group. When $Y$ is disconnected, the instanton Floer homology group can still be defined
if we require that each connected component is admissible.

Let $\Sigma$ be a submanifold of $Y$. If $\Sigma$ is disjoint from $K$, then an operator $\mu(\Sigma)$ of degree $(4-\dimm \Sigma)$ on
$\II(Y,K,\omega;\mathbb{C})$ can be defined.
We follow the convention in \cite{KM:suture, DK}.
Roughly speaking, $\mu(\Sigma)$ is defined by evaluating the class $-\frac{1}{4} p_1(\mathbb{P})/[\Sigma]$
on the moduli spaces. We use $\mathbb{P}$ to denote the universal $SO(3)$-bundle over
$\mathcal{M}\times (Y\setminus K)$ where $\mathcal{M}$ is the moduli space of trajectories of the Chern-Simons functional.
Since $-\frac{1}{4} p_1(\mathbb{P})$ is not always an integral class, we use complex coefficients
whenever we are studying the operator actions.
In most situations of this paper, $\Sigma$ will be an embedded surface. Now suppose $\Sigma$ is an embedded surface intersecting
$K$ transversely. In this case, $\mu(\Sigma)$ depends on the choice of extension of 
$\mathbb{P}\to \mathcal{M}\times (\Sigma\setminus \Sigma\cap K)$ to
$\mathcal{M}\times \Sigma$. At each point in $\Sigma\cap K$, there are two different choices of extensions. To obtain a well-defined operator, we
first pick one choice at each point to define an operator, then reverse all the choices to define another operator and finally take the average 
of these two operators. 
We denote the last operator 
by $\muu(\Sigma)$. The operator $\muu(\Sigma)$ is well-defined and does not depend on the choice of extension. 
See Section 2.3 in \cite{Street}
for more details and notice that our convention is different from \cite{Street}: 
the operators in \cite{Street} are defined by $p_1$ of the universal bundle instead of our $-\frac{1}{4}p_1$.
All those operators only depend on the homology class of $\Sigma$ in $Y$.

A cobordism $(W,S,\omega)$ between two admissible triples
$(Y,K_1,\omega_1)$ and $(Y_2,K_2, \omega_2)$ induces a homomorphism
$$
\II(W,S,\omega):\II(Y,K_1,\omega_1) \to \II (Y_2,K_2, \omega_2)
$$
which is only well-defined up to an overall sign in general. This sign ambiguity can be resolved with some extra requirements.
We will deal with this later when we need it.

Let $\mathcal{K}_m\subset S^1\times S^2$ be the submanifold  $S^1\times \{p_1,\cdots,p_m\} $ where $p_1,\cdots, p_m$ are $m$ points in $S^2$. Let
$R:=\{\text{pt}\}\times S^2\subset S^1\times S^2$
be a copy of the $S^2$-slice and $x\in S^1\times S^2\setminus \mathcal{K}_m$. 
It is clear that $R\cap \mathcal{K}_m$ consists of $m$ points. Therefore
when $m$ is odd, the triple
$(S^1\times S^2,\mathcal{K}_m,\emptyset)$ 
is admissible and $\II(S^1\times S^2,\mathcal{K}_m,\emptyset)$ can be defined.
Street studied 
 the singular instanton Floer homology of the pair $(S^1\times S^2, S^1\times \{p_1,\cdots,p_k\})$ and obtained the following.
\begin{PR}[{{\cite{Street}*{Corollary 2.9.11}}}] \label{eigenvalue}
Suppose $m\ge 3$. The only eigenvalue of $\mu(x)$ on  
$\II(S^1\times S^2,\mathcal{K}_m,\emptyset;\mathbb{C})$ is $2$. 
The eigenvalues of $\muu(R)$ on $\II(S^1\times S^2,K_m,\emptyset;\mathbb{C})$ are
$$
\mathcal{C}_m:=\{-(m-2),-(m-4),\cdots, (m-4), (m-2)\}
$$
Moreover, the generalized eigenspaces for the top eigenvalues $\pm (m-2)$ are one-dimensional.
\end{PR}
One immediate corollary of this proposition is the following 
(cf. \cite[Corollary 7.2]{KM:suture}).
\begin{COR}\label{specbound}
Suppose $(Y,K,\omega)$ is an admissible triple and $R\subset Y$ is an embedded sphere. If $|R\cap K|=m\ge 3$ is odd, then the
only eigenvalue of $\mu(x)$ is $2$ and
the eigenvalues of
$\muu(R)$ on $\II(Y,K,\omega;\mathbb{C})$ are contained in the set $\mathcal{C}_m$. Moreover, the generalized eigenspace of $\muu(R)$
with eigenvalue $l$ is isomorphic to the generalized eigenspace with eigenvalue $-l$.
\end{COR}
The first part is just \cite[Lemma 3.2.3]{Street}.
The second part is because $\muu(R)$ is a degree $2$ operator on a
$\mathbb{Z}/4$ graded space: the isomorphism is given by $\alpha\mapsto (\sqrt{-1})^{\deg\alpha}\alpha$ where $\alpha$ is a homogeneous element in the Floer homology group.

\begin{DEF}
Let $(Y,K,\omega)$ be an admissible triple and $R\subset Y$ be an embedded 2-sphere. Suppose $R$ intersects $K$ transversely at $m$ points. We define
\begin{equation*}
  \II(Y,K,\omega|R)\subset \II(Y,K,\omega;\mathbb{C})
\end{equation*}
to be the generalized eigenspace for $\muu (R)$ for the eigenvalue $(m-2)$. More generally, if $\Sigma\subset Y$ is an embedded surface,
we use  
\begin{equation*}
  \II(Y,K,\omega)_{\muu(\Sigma),n}\subset \II(Y,K,\omega;\mathbb{C})
\end{equation*}
to denote the generalized eigenspace for the operator $\muu(\Sigma)$ for the eigenvalue $n$. The operator $\muu(\Sigma)$ can also be 
replaced by $\mu(\Sigma)$ if $\Sigma\cap K=\emptyset$ or $\mu(x)$ where $x\in Y\setminus K$. 
\end{DEF}
In this definition we do not require $m$ to be odd.

Let $(Y_1, K_1,\omega_1)$ and $(Y_2,K_2,\omega_2)$ be two admissible triples. Suppose $R_i\subset Y_i$ ($i=1,2$) is an embedded sphere such that
$R_1\cdot \omega_1 =R_2\cdot \omega_2$ mod $2$ and $|R_1\cap K_1|=|R_2\cap K_2|=m$ is an odd number. All the intersections are supposed to be transversal.
Notice that if $|R_i\cap K_i|$ is odd, then $R_i$ is non-separating. Choose a diffeomorphism $h:(R_1,R_1\cap K_1)\to (R_2, R_2\cap K_2)$. Cut $Y_i$ along $R_i$ ($i=1,2$) and
identify the four boundary components by $h$ then we obtain a new admissible triple $(\widetilde{Y},\tilde{K},\widetilde{\omega})$. There are two embedded
spheres $\widetilde{R}_1$ and $\widetilde{R}_2$ which are homologous. We pick one of them and denote it by $\widetilde{R}$. Street's excision theorem is
the following.
\begin{THE}[{{\cite{Street}*{Theorem 3.2.4}}}]\label{Sexcision}
If $(\widetilde{Y},\tilde{K},\widetilde{\omega},\widetilde{R})$ is obtained as above, then
\begin{equation*}
 \II(\widetilde{Y},\tilde{K},\widetilde{\omega}|\widetilde{R})\cong 
 \II(Y_1, K_1,\omega_1|R_1)\otimes \II(Y_2, K_2,\omega_2|R_2)
\end{equation*}
\end{THE}
The isomorphism is induced by a cobordism. Part of the cobordism
is the product of $(S^2, \{m~\text{points}\})$ and the saddle surface in Figure \ref{excision-co}.
In Figure \ref{excision-co}, the two red dots represent $R_1$ and $R_2$ and the two green dots represent $\widetilde{R}_1$ and  $\widetilde{R}_2$. In the complement of 
Figure \ref{excision-co}, the cobordism is the product cobordism.

This theorem is an analogue of \cite{KM:suture}*{Theorem 7.7} where they do excision along genus $g$ surfaces (assuming $K$ is empty). The excision theorem along
tori is due to Floer \cite{Floer-surgery}. The proof of this theorem relies on the one-dimensionality of the top (generalized) eigenspace, as stated in
Proposition \ref{eigenvalue}.

\begin{figure}
\centering
\begin{tikzpicture}

\draw[thick] (-1,0) to [out=280,in=180] (0,-2) to [out=0,in=260] (1,0);

\draw[thick] (-2,1) to [out=350,in=90] (-1,0) to [out=270,in=30] (-2.5,-1.5) to   (-2.5,-4);

\draw[thick] (-2.5,-4) to [out=15,in=165] (2.5,-4) ;

\draw[thick] (2,1) to [out=190,in=90] (1,0) to [out=270,in=150] (2.5,-1.5) to    (2.5,-4);

\draw[dashed, thick,dash pattern=on 0.08cm off 0.1cm] (-2,-2.5) to [out=350,in=190](2,-2.5);

\draw[thick,dash pattern=on 2.23cm off 0.1cm on 0.1cm off 0.1cm on 0.1cm off 0.1cm on 0.1cm off 0.1cm on 0.1cm off 0.1cm on 0.1cm
 off 0.1cm on 0.1cm off 0.1cm on 0.1cm] (-2,1) to  (-2,-2.5);
\draw[thick,dash pattern=on 2.23cm off 0.1cm on 0.1cm off 0.1cm on 0.1cm off 0.1cm on 0.1cm off 0.1cm on 0.1cm off 0.1cm on 0.1cm
 off 0.1cm on 0.1cm off 0.1cm on 0.1cm off 0.1cm ] (2,1) to  (2,-2.5);

\draw[fill, red] (0,-3.63) circle (1pt); \node[below] at (0,-3.63) {$R_1$};
\draw[fill, red] (0.05,-2.685) circle (1pt); \node[below] at (0.05,-2.685) {$R_2$};
\draw[fill, green] (1,0) circle (1pt);\node[left] at (1,0) {$\widetilde{R}_2$};
\draw[fill, green] (-1,0) circle (1pt);\node[right] at (-1,0) {$\widetilde{R}_1$};

\end{tikzpicture}
\caption{A schematic picture for the excision cobordism}\label{excision-co}
\end{figure}
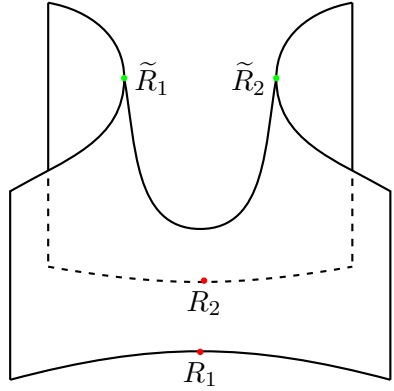

 Let $u_i$ be an arc in a $S^2$-slice in $S^1\times S^2$
that  joins two circles in $\mathcal{K}_m$. 
The triple $(S^1\times S^2, \mathcal{K}_m, u_1+\cdots u_l)$
is still admissible if we assume each circle in $\mathcal{K}_m$ is touched by at most one arc. When $m$ is odd,
we can apply the excision theorem to two copies of this triple and obtain the following.
\begin{PR}\label{S1S2u}
We have
\begin{equation*}
  \II(S^1\times S^2, \mathcal{K}_m, u_1+\cdots u_l|R)\cong \II(S^1\times S^2, \mathcal{K}_m, \emptyset|R)=\mathbb{C}
\end{equation*}
where $R$ is a $S^2$-slice disjoint from those arcs $u_1,\cdots, u_l$ and $m$ is an odd number.
\end{PR}
From now on we fix an arc $u$ joining two connected components of $\mathcal{K}_m$. The critical points
of the Chern-Simons functional for the triple $(S^1\times S^2, \mathcal{K}_2,u)$ consist of $SU(2)$ representations (modulo conjugacy)
of $\pi_1(S^1\times S^2\setminus(\mathcal{K}_2\cup u))$
satisfying that the holonomy on the meridians of $\mathcal{K}_2$ is conjugate to $\mathbf{i}\in SU(2)$ and the holomony on the meridian
of $u$ is $-1$.
The complement of $\mathcal{K}_2$ in $S^1\times S^2$ is $S^1\times D_1$
where $D_1$ is an open disk punctured at the origin. We also assume the arc $u$ 
lies in a $D_1$-slice
 and goes from the origin to the boundary. Let $c$ be a circle in $D_1$ around
 the origin which intersects $u$ at a single point. 
The complement of $u$ in $S^1\times D_1$ deformation retracts to the punctured
torus $S^1\times c -\{\pt\}$. The fundamental group of $S^1\times c -\{\pt\}$
is a free group of two generators $[S^1\times \{\pt\}]$ and $[c]$.
Suppose $J_1$ and $J_2$ are holonomies of the two generators respectively.
Then our holonomy condition requires that
$$
[J_1,J_2]=-1.
$$
By some linear algebra one can show that up to a conjugation, we must have
$J_1=\mathbf{i}$ and  $J_2=\mathbf{j}$. Therefore this is a unique critical 
point in this situation.

Next we want to show that this unique critical point $\rho$ is non-degenerate.
According to \cite[Lemma 3.13]{KM:YAFT}, the kernel of  
$\text{Hessian}_{\rho}(CS)$ is isomorphic to the kernel of the map
$$
H^1(S^1\times S^2-\mathcal{K}_2,\text{ad}\rho)
\to H^1(m_1\cup m_2, \text{ad}\rho|_{m_1\cup m_2})
$$ 
where $m_1,m_2$ are meridians of the two components of $\mathcal{K}_2$ and 
$\text{ad}\rho$ is the associated $\mathfrak{su}(2)$ bundle
viewed as a local system. It is a direct calculation to show that this kernel is zero. Hence
$\rho$ is non-degenerate. 
Therefore $\II(S^1\times S^2, \mathcal{K}_2, u)\cong\mathbb{Z}$ and $\mu(\{\pt\}\times S^2)$ is $0$ since it is a degree 2 operator
on a $\mathbb{Z}/4$-graded space.

In general, the complement of $\mathcal{K}_m$ is $S^1\times D_{m-1}$ where $D_{m-1}$ is an open disk punctured at $(m-1)$ points. In this case, there
are $2^{m-2}$ critical points. $S^1\times \{\pt\}$ and the small circle around the origin are still mapped to $\mathbf{i}$
and $\mathbf{j}$ respectively. The small circles
around the new punctures are mapped to $\pm \mathbf{i}$. We use $R_m$ to denote the $S^2$-slice in $(Y,\mathcal{K}_m,u)$.
From Proposition \ref{S1S2u} and the symmetry of eigenvalues we know the eigenvalues of
$\muu(R_3)$ are $\pm 1$. This implies the degrees of two critical points described above differ by $2$. Otherwise
$\muu(R_3)$ must be $0$ as an operator of degree $2$. In particular, there is no differential in the chain complex, we have
$$
\II(S^1\times S^2, \mathcal{K}_3, u)\cong\mathbb{Z}^2
$$
The two generators are interchanged by the diffeomorphism which is defined by reflecting the $S^1$ factor.

We also need to use another type of excision theorem. Let $(Y,K,\omega)$ be an admissible triple with two connected components
 and $T_1, T_2$ be
  two disjoint separating tori  which are in different components of $Y$, 
  disjoint from $K$ and have non-zero pairing with $\omega$ ($\text{mod}~2$). Pick a diffeomorphism $h:T_1\to T_2$.
Cut along those two tori and
re-glue by $h$ to obtain a new admissible triple $(\widetilde{Y}, \widetilde{K},\widetilde{\omega})$. The following theorem is a special case of
\cite{KM:Kh-unknot}*{Theorem 5.6}:
\begin{THE}\label{Texcision}
Let $(Y,K,\omega)$ and $(\widetilde{Y}, \widetilde{K},\widetilde{\omega})$ be given as above. Then there is a cobordism which induces an isomorphism  
\begin{equation*}
 \II(Y,K,\omega)\cong \II(\widetilde{Y}, \widetilde{K},\widetilde{\omega})
\end{equation*}
\end{THE}
\begin{rmk}
When $Y$ is connected, the above theorem still holds if we use complex coefficients and take the generalized 
eigenspace for $\mu(x)$ with eigenvalue $2$ on the right hand side. Indeed, this is
a special case of Theorem 7.7 in \cite{KM:suture}.
\end{rmk}
Pick a boundary torus $T$ of the regular neighborhood of one component of $\mathcal{K}_3$ that is touched by $u$.
Then we have $u\cdot T=1$. 
Take two copies of the triple $(Y,\mathcal{K}_3,u)$, cut along the two copies of $T$ and re-glue, we obtain two
admissible triples $(S^1\times S^2, \mathcal{K}_4, u)$ and $(S^1\times S^2, \mathcal{K}_2, u)$. By Theorem \ref{Texcision}, we have
\begin{equation*}
  \II(S^1\times S^2, \mathcal{K}_4, u)\otimes \II(S^1\times S^2, \mathcal{K}_2, u)\cong \II(S^1\times S^2, \mathcal{K}_3, u)\otimes \II(S^1\times S^2, \mathcal{K}_3, u)
\end{equation*}
So we conclude $\II(S^1\times S^2, \mathcal{K}_4, u)\cong\mathbb{Z}^4$. Moreover, the above isomorphism is induced by a cobordism from
$(S^1\times S^2, \mathcal{K}_4, u)\sqcup (S^1\times S^2, \mathcal{K}_2, u)$ to 
$(S^1\times S^2, \mathcal{K}_3, u)\sqcup (S^1\times S^2, \mathcal{K}_3, u)$
(see \cite{KM:Kh-unknot}*{Theorem 5.6} for more details). 
In this cobordism,
$[R_4]+[R_2]$ is homologous to $[R_3]+[R_3']$ where $R_3$ and $R_3'$ denote the $S^2$-slices in the two copies of $(S^1\times S^2, \mathcal{K}_3, u)$ respectively. Hence
the isomorphism (after complexifying both sides) intertwines the map $\muu(R_4)\otimes 1 +1\otimes \muu (R_2)$ on the incoming end with
\begin{equation*}
  \muu(R_3)\otimes 1 +1\otimes \muu (R_3)
\end{equation*}
on the outgoing end. Instead of using $R_3'$,
we abuse the notation $R_3$  when it does not cause any confusion. Since $\muu (R_2)=0$, the eigenspace decomposition of
$\mu(R_4)$ is also clear from the isomorphism. In particular the top  generalized eigenspace (with eigenvalue $\pm 2$ in this case) is again 1-dimensional. 
Iterating this
excision procedure, we have
\begin{equation*}
  \II(S^1\times S^2, \mathcal{K}_m, u)\cong \II(S^1\times S^2, \mathcal{K}_3, u)^{\otimes (m-2)}\cong\mathbb{Z}^{2^{(m-2)}}
\end{equation*}
This isomorphism (again after complexifying both sides) intertwines $\muu(R_m)$ with
\begin{equation*}
  \muu(R_3)\otimes 1 \otimes \cdots \otimes 1 + \cdots  + 1 \otimes \cdots \otimes 1 \otimes \muu(R_3)
\end{equation*}
So eigenspace decomposition is also clear and the top generalized eigenspace (with eigenvalue $\pm (m-2)$) is 1-dimensional. This time $m$ could be even.

\subsection{Tangle Floer homology}
The following definition of instanton Floer homology for balanced admissible tangles is 
due to Street \cite[Definition 3.3.5]{Street}. 
\begin{DEF}
Suppose $T\subset I\times D$ is a balanced admissible tangle. Glue another pair of $(I\times D, I\times \{p_1,p_2\})$ to $(I\times D, T)$
by identifying $\partial D\times I$ in the obvious way, we
obtain a pair $(I\times S^2, T\sqcup I\times \{p_1,p_2\})$. Identify the top and bottom 2-spheres
(denoted by $R^{\pm}$) by an arbitrary diffeomorphism that closes up $T$ into a link $\hat{T}$
and $I\times \{p_1,p_2\}$ into product circles $S^1\times \{p_1,p_2\}$,
the resulting pair is
 $(S^1\times S^2, \hat{T}\sqcup \mathcal{K}_2)$.
 The instanton Floer homology for the tangle $T$ is defined as
 \begin{equation*}
   \THI (T):=\II(S^1\times S^2, \hat{T}\sqcup \mathcal{K}_2, u|R)
 \end{equation*}
where  $u$ is an arc joining two connected components of
$\mathcal{K}_2$ as before and $R$ is the resulting 2-sphere after identifying $R^{\pm}$. When
$|T\cap D^+|$ is odd, $(S^1\times S^2, \hat{T},\emptyset)$ is admissible.
The \emph{odd tangle Floer homology} of $T$ can be defined as
 \begin{equation*}
   \THI^{\text{odd}} (T):=\II(S^1\times S^2, \hat{T}, \emptyset)
 \end{equation*}
\end{DEF}
When $|T\cap D^+|=|T\cap D^-|=l$ is odd, using Theorem \ref{Sexcision} and the fact $\II(S^1\times S^2,\mathcal{K}_{l+2},\emptyset|R_{l+2})$
is 1-dimensional, one can show that $\THI(T)$ (or $\THI^{\text{odd}} (T)$)
 does not depend on the choice of the diffeomorphism $h$ as in \cite{KM:suture}*{Corollary 4.8}. 
If $l$ is even,
we can apply the torus excision to $$(S^1\times S^2, \hat{T}\sqcup \mathcal{K}_2, u)\sqcup (S^1\times S^2, \mathcal{K}_3, u) $$
as we did in the calculation of $\II(S^1\times S^2, \mathcal{K}_m, u)$ to obtain
\begin{equation*}
  \II(S^1\times S^2, \hat{T}\sqcup \mathcal{K}_2, u|R)\cong \II(S^1\times S^2, \hat{T}\sqcup \mathcal{K}_3, u|R')
\end{equation*}
where both $R$ and $R'$ are the $S^2$-slices.
Then the problem is reduced to previous case.

Next we state a non-triviality result due to Street.
\begin{THE}\label{Street-detection}
Let $T\subset I\times D$ be a vertical  balanced admissible tangle. Then $\THI(T)\cong \mathbb{C}$ 
if and only if there is an orientation-preserving diffemorphism of pairs
\begin{equation*}
  F:(I\times D, I\times \{p_1,\cdots, p_m\} ) \to (I\times D, T)
\end{equation*}
such that $F(D^{\pm})=D^{\pm}$ and $F(I\times \partial D)=I\times \partial D$.
\end{THE}
Our statement here is not exactly the same as Theorem 3.4.4 in \cite{Street}, in which he concludes that $T$ is a product tangle if $\THI(T)=\mathbb{C}$.
What he actually proves is that
the complement of a tubular neighborhood of $T$ is a product manifold, which is equivalent to our statement. 
The tangle $T$ is the same as the product tangle
modulo diffeomorphisms on $I\times D$.
In this paper, we want to define
 two tangles to be equivalent if there is an ambient isotopy which moves one tangle to another one and
\emph{fixes the end points of the tangle all the time}. It may happen that $F$ 
represents a non-trivial element in the mapping class group
$$\text{MCG}(I\times D, \partial I \times \{p_1,\cdots, p_m\}).$$
So we cannot conclude $T$ is a product tangle in our sense. Instead, we have the following.
\begin{COR}\label{braid-detection}
Let $T\subset I\times D$ be a vertical  balanced admissible tangle. Then $\THI(T)\cong\mathbb{C}$ if and only if  $T$ is isotopic to a braid.
\end{COR}
\begin{proof}
By Theorem \ref{Street-detection}, it suffices to show that there is a diffeomorphism 
\begin{equation}\label{F-map-T}
  F:(I\times D, I\times \{p_1,\cdots, p_m\} ) \to (I\times D, T)
\end{equation}
if and only if $T$ is isotopic to a braid.

The braid group $B_m$ is isomorphic to the mapping class group 
\begin{equation*}
M_m:=\text{MCG}(D\setminus \{p_1,\cdots, p_m \}, \partial D)
\end{equation*}
of a $m$-punctured disk rel the boundary. To be more precise, the isomorphism $M_m\to B_m$ can be constructed as follows.
Given an element $[f]\in M_m $, the map $f$ can be extended into a 
map $f: D\to D$ which is the identity map on the boundary. 
There is an isotopy $h_t : D\to D$ with $h_{1}=\idd$, $h_{-1}=f$ and $h_t|_{\partial D}=\idd$. 
Then a braid in $B_m$ can be defined by
\begin{equation}\label{isotopy-braid}
 \{(t,h_t(p_i))|t\in I, i=1,\cdots,n \} \subset I\times D
\end{equation} 

If $T$ is a braid, then the isotopy $h_t$ is the diffeomorphism $F$ we need. On the other hand, if we have a diffeomorphism $F$
in \eqref{F-map-T}, by composing with $(\idd_I\times F_1)^{-1}$ where $F_1:D\to D$ is defined as $F|_{\{1\}\times D}$, we obtain
a diffeomorphism ${F}':I\times D\to I\times D$ whose restriction on $\{1\}\times D$ is the identity map. The map 
${F}'$ is isotopic to a diffeomorphism $\widetilde{F}$
which is the identity map on $I\times \partial D\cup \{1\}\times D$ and preserves the $I$-coordinate. 
Moreover, this isotopy can be chosen to be identity on $\{1\}\times D\cup \{-1\}\times D'$ where $D'$  is a smaller disk in $D$ but still
contains $p_1,\cdots, p_m$. We can define an isotopy $h_t:=F_{\{t\}\times D}$. 
From the construction we see that $T$ is isotopic to the braid defined by \eqref{isotopy-braid}.
\end{proof}

To illuminate Street's techniques, we prove a detection result for odd tangle Floer homology, whose proof is adapted from
the proof of  \cite[Theorem 3.4.4]{Street}. 
\begin{THE}\label{odd-braid-detection}
Let $T\subset I\times D$ be a vertical  balanced admissible tangle such that $|T\cap D^+|=|T\cap D^-|=m$ is odd and no less than $3$. We also embed
$I\times D$ into $I\times S^2$ in the obvious way. 
If $\THI^{\emph{odd}}(T)\cong\mathbb{C}$, then $T$ is isotopic to a braid through
an isotopy of $I\times S^2$.
\end{THE}
\begin{proof}
We define 
a new tangle
$T'$ as the union of $T$ and a parallel copy of a component of $T$. Then we define
a sutured manifold  $M$  as $I\times S^2-N(T')$. 
We will use the excision theorems (Theorem \ref{Sexcision} and Theorem \ref{Texcision})
to show that $\THI^{\text{odd}}(T)$ is isomorphic to the instanton Floer homology of a sutured manifold $M$.
Then the one-dimensionality
of the sutured Floer homology implies $M$ is a product sutured manifold. Then
it follows from the construction of $M$ that $T$ is isotopic to a braid.

Suppose $\THIo(T)\cong\mathbb{C}$ and $m=2l+1$. 
Close up  each component of $T$ with itself and denote the resulting
link in $S^1\times D$ by $\hat{T}$.
Denote the components of the closure $\hat{T}$ by $L_1,\cdots, L_m$.  Apply Theorem \ref{Sexcision} and
Proposition \ref{S1S2u} to 
\begin{equation*}
(S^1\times S^2,\hat{T},\emptyset) \sqcup (S^1\times S^2, \mathcal{K}_m,u)
\end{equation*} 
where $u$ is an arc joining the first and second components of $\mathcal{K}_m$, we obtain 
\begin{equation*}
 \II(S^1\times S^2,\hat{T}, w|R) \cong \THIo(T)\cong \mathbb{C}
\end{equation*}
where $w$ is an arc joining $L_{1}$ and $L_{2}$. Do excision to 
\begin{equation*}
(S^1\times S^2,\hat{T},w) \sqcup (S^1\times S^2, \mathcal{K}_4,u)
\end{equation*} 
along tori $T_1$ and $T_2$ where $T_1=\partial N(L_1)$ and $T_2$  is the boundary torus of a tubular neighborhood of
the first component of $\mathcal{K}_4$,
we obtain 
\begin{equation*}
(S^1\times S^2,L_1^3\cup L_2\cup\cdots \cup L_m,w) \sqcup (S^1\times S^2, \mathcal{K}_2,u)
\end{equation*} 
where $L_1^3=L_{1,1}\cup L_{1,2}\cup L_{1,3}$ 
is the 3-cable of $L_1$ and $w$ joins  $L_{1,3}$ and $L_2$. By Theorem \ref{Texcision},
we have
\begin{align*}
\II(S^1\times S^2,L_1^3\cup L_2\cup\cdots L_m,w)&\otimes 
\II (S^1\times S^2, \mathcal{K}_2,u)    \\
\cong\II(S^1\times S^2,\hat{T},w) &\otimes  \II(S^1\times S^2, \mathcal{K}_4,u)
\end{align*}
The above isomorphism intertwines 
the operator $\muu(S^2)\otimes 1 +1\otimes \muu(S^2)$ on the top part
with the operator $\muu(S^2)\otimes 1 +1\otimes \muu(S^2)$ on the bottom part,
since the above isomorphism is induced by the excision cobordism
and homology classes used to define the two operators are homologous in the 
cobordism. Taking the top generalized eigenspaces of the two operators, we obtain
\begin{equation*}
\II(S^1\times S^2,L_1^3\cup L_2\cup\cdots \cup L_m,w|R) \cong \II(S^1\times S^2,\hat{T}, w|R) \cong \mathbb{C}
\end{equation*} 
Let $u_i$ ($1\le i\le l$) be an arc joining the $(2i+2)$-th and $(2i+3)$-th components of $\mathcal{K}_{m+2}$.
 Let $u_0$ be an arc joining the first two components of $\mathcal{K}_{m+2}$ and $u'$ be an arc joining
 the third and fourth components of $\mathcal{K}_{m+2}$.
Apply Theorem \ref{Sexcision} and Proposition \ref{S1S2u} to 
\begin{equation*}
(S^1\times S^2,L_1^3\cup L_2\cup\cdots \cup L_m,w) \sqcup (S^1\times S^2,\mathcal{K}_{m+2}, u'+u_0+u_1+\cdots +u_l)
\end{equation*}    
we obtain
\begin{equation*}
\II(S^1\times S^2,L_1^3\cup L_2\cup\cdots \cup L_m,(w_0+w_1+\cdots+w_l)|R) \cong \mathbb{C}
\end{equation*}
where $w_i$ ($1\le i\le l$) is an arc joining $L_{2i}$ and $L_{2i+1}$ while $w_0$ is an arc joining $L_{1,1}$ and $L_{1,2}$.
Now do excision to
\begin{equation*}
(S^1\times S^2,L_1^3\cup L_2\cup\cdots \cup L_m,(w_0+w_1+\cdots+w_l)) \sqcup (S^1\times S^2, \mathcal{K}_2, u)
\end{equation*}
along tori $T_1$ and $T_2$ where $T_1=\partial N(L_{1,2})$ with the additional requirement that $L_{1,3}\subset N(L_{1,2})$ 
 and $T_2$ is the boundary torus of a tubular neighborhood of one component of $\mathcal{K}_2$, we obtain
\begin{equation*}
(S^1\times S^2,L_{1,1}\cup L_{1,2}\cup L_2\cup\cdots \cup L_m,(w_0+w_1+\cdots+w_l)) \sqcup (S^1\times S^2, \mathcal{K}_3, u)
\end{equation*}
Using Theorem \ref{Texcision} and taking proper generalized eigenspaces, we have
\begin{equation*}
\II(S^1\times S^2,\hat{T}',(w_0+w_1+\cdots+w_l)|R) \cong  \mathbb{C}
\end{equation*}
where we use $T'$ to denote the tangle obtained by replacing the first component of $T$ by its 2-cable so that we have
$\hat{T}'=L_{1,1}\cup L_{1,2}\cup L_2\cup\cdots \cup L_m$.

Now $\hat{T}'$ is a link with $(2l+2)$ components and the components are divided into pairs joined by arcs. 
Let $T_1,\cdots, T_{2l+2}$ be the boundary tori of 
$$N(L_{1,1}),N(L_{1,2}),N(L_2),\cdots, N(L_m)$$ 
Do excision along 
pairs of tori $T_{2i+1},T_{2i}$ ($0\le i \le l+1$), we obtain a triple $(Y(T'), \emptyset,w')$ and $(l+1)$ copies of 
$(S^1\times S^2,\mathcal{K}_2,u)$. The 3-manifold $Y(T')$ can be constructed by gluing $(l+1)$ copies 
of $S^1\times A$ ($A$ is an annulus) to 
$$S^1\times S^2\setminus N(L_{1,1})\cup\cdots\cup N(L_m)$$
By Theorem \ref{Texcision} (and the remark below it),
we have
\begin{equation*}
\II(Y(T'),  \emptyset,w')_{\mu(x),2}\cong \II(S^1\times S^2,\hat{T}',(w_0+w_1+\cdots+w_l))
\end{equation*}
After the excision, the sphere $R\subset S^1\times S^2$ becomes a genus $(l+1)$ surface $\Sigma\subset Y(T')$, we have
\begin{equation*}
\II(Y(T'),\emptyset, w'|\Sigma)\cong \II(S^1\times S^2,\hat{T}',(w_0+w_1+\cdots+w_l)|R)\cong\mathbb{C}
\end{equation*}
where $\II(Y(T'),\emptyset, w'|\Sigma)$ is defined as the simultaneous generalized eigenspace of operators $\mu(\Sigma)$
and $\mu(x)$ with eigenvalues $2l$ and $2$ respectively. Now we view $T'$ as a tangle in $I\times S^2$ by gluing another copy of
$I\times D$ to $I\times D$ along $I\times \partial D$. Let $M$ be the sutured manifold $M=I\times S^2\setminus N(T')$ with
a meridian suture for each component of $T'$. According to \cite[Lemma 2.3]{KM:Alexander}, 
$\II(Y(T'), w' ,\emptyset|\Sigma)$ is isomorphic to the instanton Floer homology $\SHI(M)$ defined for sutured manifolds. 

Next we want to apply \cite{KM:suture}*{Theorem 7.18}
to show that $M$ is a product sutured manifold. For this purpose, we need to show
that $M$ is a balanced taut sutured manifold and a homology product.
Since $T'$ is a vertical tangle, it is clear
from the definition that $M$ is balanced.  
Since $I\times S^2$ and $N(T')$ are product manifolds, it is a direct application 
of the long exact sequence in homology to show that $M=I\times S^2-N(T')$ is a 
homology product. 

The ``top surface''$R^+$ 
(resp. the ``bottom surface'' $R^-$) of $M$
is a punctured 2-sphere  whose punctures are in one-to-one correspondence with 
the sutures. The boundary components of 
any other surface in the homology class of $R^+$ (resp. $R^-$)
is no less than the number of sutures. This implies that  
$R^+$ (resp. $R^-$) minimizes the Thurston norm. 
Suppose  $R^+$ (resp. $R^-$) is compressible.
 Compressing $R^+$ (resp. $R^-$) along the compressing disk would increase the 
 Euler characteristic. The Thurston norm will decrease 
 unless $R^+$ (resp. $R^-$) is an annulus. Since  
 $R^+$ (resp. $R^-$) has at least 4 boundary components by our assumption,
 the Thurston norm must decrease in the compressing process, which contradicts
 the norm-minimizing property of $R^+$ (resp. $R^-$). Therefore we conclude that
 $R^+$ (resp. $R^-$) is incompressible. 
  Pick a 2-sphere $S$ in $M$. By the Light Bulb theorem 
  (see \cite[Page 257]{Rolfsen} for example), the complement of a component
  of $T'$ in $I\times S^2$ is homeomorphic to $I\times D$. Since $I\times D$
  is irreducible, $S$ bounds a 3-ball $B$ in the complement of a component of 
  $T'$. Since $T'$ is embedded, $B$ is disjoint from $T'$. Therefore 
  $S$ bounds a 3-ball in $M$. This means $M$ is irreducible.
In summary, we show that $M$ is taut.

Now by \cite{KM:suture}*{Theorem 7.18}, $\SHI(M)\cong\mathbb{C}$ implies 
$M$ is the product sutured manifold $I\times F$ where $F$ is a punctured $S^2$.  From the proof of 
Corollary \ref{braid-detection} we see that $T'$ is isotopic to a braid . 
Therefore $T$ is also isotopic to a braid. 
\end{proof}

\section{Instanton Floer Homology for Links in a Thickened Annulus}\label{AHI}
In this section,
we will define the counterpart of the annular Khovanov homology in the instanton Floer homology
setting. We use $S^1\times D$ to represent a thickened annulus.
\subsection{Definition and basic properties}
 \begin{DEF}\label{AHI-def}
Let $L\subset S^1\times D$ be a link. Take another pair $(S^1\times D,\mathcal{K}_2)$ where $\mathcal{K}_2=S^1\times \{p_1,p_2\}$ and
glue the two copies of $S^1\times D$ by identifying the boundary in the obvious way to obtain a new pair $(S^1\times S^2, L\sqcup \mathcal{K}_2)$.
The annular instanton Floer homology for $L$ is defined as
\begin{equation*}
  \AHI(L):=\II(S^1\times S^2, L\sqcup \mathcal{K}_2, u)
\end{equation*}
where $u$ is an arc joining two connected components of $\mathcal{K}_2$ as before.
\end{DEF}

\begin{eg}\label{AHIU1K1}
By abuse of notation, we use $\mathcal{K}_1$ to denote the homologically non-trivial circle $S^1\times \{\pt\}\subset S^1\times D$. Then we have
$\AHI(\mathcal{K}_1)=\mathbb{Z}^2$. We use $U_1$ to denote an unknot in $S^1\times D$, then we have $\AHI(U_1)=\mathbb{Z}^2$. This can be seen from fact that
the Chern-Simons functional
for the triple $(S^1\times S^2, U_1\sqcup \mathcal{K}_2, u)$ has Morse-Bott critical manifold $\mathbb{CP}^1$ (c.f. Proposition 4.10 in \cite{KM:YAFT}).
Alternatively, one can apply the torus excision to cut along the boundary of the neighbourhood of a circle in $\mathcal{K}_2$ and re-glue the boundary by
interchanging its meridian and longitude. In this way, we obtain a new triple $(S^3, U_1\sqcup H,u)$ where $H$ is a Hopf link in $S^3$ and $u$ is an arc joining
the two components of $H$. Now we have $\AHI(U_1)\cong \II^\sharp (U_1)=\mathbb{Z}^2$ where $\II^\sharp$ is the invariant defined in \cite{KM:Kh-unknot}.
See Section \ref{re-inst} for more details.
\end{eg}

\bigbreak

Suppose $L_1$ and $L_2$ are two links in $S^1\times D$. Without loss of generality, we may assume $L_1, L_2$ lie in $S^1\times D_1$ and $S^1\times D_2$ respectively
where $D_1$ and $D_2$ are two disjoint disks in $D$. We use $L_1\sqcup L_2$ to denote the new link formed as the union of $L_1$ and $L_2$.
\begin{PR}\label{link-sum}
Let $L_1\sqcup L_2$ be obtained as above, then
\begin{equation*}
  \AHI(L_1\sqcup L_2)=\II( (S^1\times S^2, L_1\sqcup \mathcal{K}_2, u)\sqcup  (S^1\times S^2, L_2\sqcup \mathcal{K}_2, u))
\end{equation*}
In particular, if one of $\AHI(L_1)$ and $\AHI(L_2)$ is torsion-free, then we have
\begin{equation}\label{L1+L2}
  \AHI(L_1\sqcup L_2)\cong\AHI(L_1)\otimes \AHI(L_2)
\end{equation}
Moreover, this isomorphism (after complexifying both sides)
intertwines $\muu(R)$ with $\muu(R_1)\otimes 1 +1\otimes \muu(R_2)$ where $R, R_1, R_2$ are all $S^2$-slices in $S^1\times S^2$.
\end{PR}
\begin{proof}
The proof is based on the same method that we used to calculate $\II(S^1\times S^2, \mathcal{K}_m,u)$. 
Let $T$ be the boundary torus of a tubular neighborhood of one component of
$\mathcal{K}_2$. Then we have two tori $S^1\times m$ in  $(S^1\times S^2, L_1\sqcup \mathcal{K}_2, u)$ and $(S^1\times S^2, L_2\sqcup \mathcal{K}_2, u)$ respectively.
Do excision along the two tori we obtain $(S^1\times S^2, L_1\sqcup L_2 \sqcup \mathcal{K}_2, u)\sqcup (S^1\times S^2,\mathcal{K}_2, u)$.
By Theorem \ref{Texcision}, 
we obtain the first statement of the proposition. The proof of the last part is the same as how we study the sphere operators
on $\II(S^1\times S^2, \mathcal{K}_m,u)$ in the previous section.
\end{proof}

If we use $U_k$ to denote the unlink with $k$ components. By Proposition \ref{link-sum} we have
\begin{equation*}
  \AHI(U_k\sqcup \mathcal{K}_l)\cong \AHI(U_1)^{\otimes k}\otimes \AHI(\mathcal{K}_1)^{\otimes l}=\mathbb{Z}^{2^{(k+l)}}
\end{equation*}

The annular instanton Floer homology has a (relative) $\mathbb{Z}/4$ homological grading. If we use complex coefficients,
 we can define another grading for it.
\begin{DEF}
We define
\begin{equation*}
  \AHI(L,i)\subset \AHI(L;\mathbb{C})= \II(S^1\times S^2, L\sqcup \mathcal{K}_2, u;\mathbb{C})
\end{equation*}
as the \emph{generalized} eigenspace for the operator $\muu(R)$ for the eigenvalue $i\in \mathbb{Z}$ where $R$ is a $S^2$-slice in $S^1\times S^2$.
\end{DEF}
Since $\muu(R)$ is an operator of degree $2$, the homological grading descends to a $\mathbb{Z}/2$-grading after we take the generalized eigenspaces.

By Proposition \ref{link-sum}, we have
\begin{equation}\label{unlink-AHI}
  \dim \AHI(U_k\sqcup \mathcal{K}_l,i)=\left\{\begin{array}{cc}
                                                2^k {l \choose (l+i)/2} & \text{if}~|i|\le l~\text{and}~l-i=0 ~(\text{mod}~2) \\
                                                0 & \text{otherwise.}
                                              \end{array}\right.
\end{equation}

Let $T\subset I\times D$ be a balanced admissible tangle such that $|T\cap D^+|=m$. We can close up
$T$ to obtain a link $\hat{T}\subset S^1\times D$. It
is clear from the definition that
\begin{equation}\label{THI=AHI}
  \THI(T)\cong \AHI(\hat{T},m)
\end{equation}

\begin{PR}\label{eigen-range}
Let $L$ be a link in $S^1\times D$. Suppose there is a slice $\{\pt\}\times D$ 
that intersects $L$ transversely at $m$ points. If $\AHI(L,i)\neq 0$ then we must have
\begin{equation*}
  i\in \mathcal{C}_{m+2}:=\{-m, -(m-2),\cdots, m-2, m\}.
\end{equation*}
We also have
\begin{equation*}
  \AHI(L,i)\cong \AHI(L,-i)
\end{equation*}
\end{PR}
\begin{proof}
Since $\AHI(L)=\II(S^1\times S^2, L\sqcup \mathcal{K}_2,u)$, the proposition follows from Corollary \ref{specbound} directly when $m$ is odd.
When $m$ is even, we use Proposition \ref{link-sum} to obtain an isomorphism
\begin{equation*}
  \AHI(L\sqcup \mathcal{K}_1)\cong \AHI(L)\otimes \AHI(\mathcal{K}_1)
\end{equation*}
Recall that $\AHI(\mathcal{K}_1)$ is supported 
at f-degrees $\pm 1$ by \eqref{unlink-AHI}.
Since this isomorphism also intertwines the $\muu(R)$-operators, we reduce the problem back into the case of odd intersection points .
The second part follows from
Corollary \ref{specbound}.
\end{proof}

\subsection{Functoriality}
Given two links $L_1,L_2$ in $S^1\times D$ and a (possibly non-orientable) cobordism $S\subset I\times S^1\times D$ from $L_1$ to $L_2$, there is a map induced by the cobordism
\begin{equation*}
  \AHI(S):\AHI(L_1)\to \AHI(L_2)
\end{equation*}
defined by
\begin{equation*}
  \AHI(S):=\II(I\times S^1\times S^2, S\sqcup I\times \mathcal{K}_2, I\times u)
\end{equation*}
This map is only well-defined up to an overall sign. Let $\mathcal{A}$ be the category of annular links with morphisms the link cobordisms.
$\AHI$ is a functor
\begin{equation*}
  \AHI: \mathcal{A}\to \mathcal{G}_p
\end{equation*}
where $\mathcal{G}_p$ is the category of abelian groups in which the  
morphisms are taken to be group homomorphisms modulo $\pm 1$.

If $L_1$ and $L_2$ are oriented and $S$ is an oriented cobordism from $L_1$ to $L_2$, then
$\AHI(S)$ can be defined without the sign ambiguity by \cite{KM:Kh-unknot}*{Section 4.4}. We denote
the category of oriented annular links with morphisms the oriented link cobordisms by $\mathcal{A}_o$, then $\AHI$ is a functor
\begin{equation*}
  \AHI: \mathcal{A}_o\to \mathcal{G}
\end{equation*}
where $\mathcal{G}$ is the category of abelian groups.

 The isomorphism \eqref{L1+L2} is natural under split cobordisms: the disjoint union of cobordisms contained in disjoint subspaces $S^1\times D_1, S^1\times D_2$ of
$S^1\times D$ respectively.

\subsection{Relations to other instanton invariants}\label{re-inst}
We want to understand the relation between the annular instanton Floer homology $\AHI$ and the singular instanton Floer homology for links in $S^3$ defined in
\cite{KM:Kh-unknot}. Let $H$ be a Hopf link in $S^3$ and $\omega$ be an arc joining the two components of $H$, then $R(S^3, H, w)$ consists of a single
non-degenerate point. Therefore we have
\begin{LE}[{{\cite{KM:Kh-unknot}*{Proposition 4.1}}}]\label{IH}
Let $H$ and $w$ be defined as above. Then
\begin{equation*}
  \II(S^3,H,w)\cong \mathbb{Z}
\end{equation*}
\end{LE}

Let $L\subset S^1\times D$ be a link.
Denote the boundary of the tubular neighbourhood of a component in $\mathcal{K}_2\subset S^1\times S^2$ by $T_1$. Similarly, pick a component of the Hopf link and
denote the boundary torus of its neighborhood by $T_2$.  Do excision to
\begin{equation*}
  (S^1\times S^2, L\sqcup \mathcal{K}_2, u)\sqcup (S^3, H, \omega)
\end{equation*}
along $T_1$ and $T_2$. Choosing the diffeomorphism between $T_1$ and $T_2$ 
which maps the meridian (longitude) of $T_1$ to the longitude (meridian) of $T_2$,
then  the new admissible triples would be
\begin{equation*}
  (S^3,\widetilde{L},u)\sqcup (S^1\times S^2, \mathcal{K}_2,\omega)
\end{equation*}
Theorem \ref{Texcision} and Lemma \ref{IH} imply
\begin{equation}\label{AHI-I}
  \AHI(L)\cong \II(S^3,\widetilde{L},u)
\end{equation}
The effect of this excision is the same as doing $0$-surgeries on the two old triples along chosen components of the links. After the excision,
the Hopf link $H$ becomes the product links in $S^1\times S^2$ and the arc $w$ still joins the two components.
$L\sqcup \mathcal{K}_2$ becomes a link $\widetilde{L}$ in $S^3$.  $\widetilde{L}$ can be described by the following way (see Figures \ref{LK} and \ref{Ltilde}):
\begin{enumerate}
  \item    Embed $S^1\times D$ into $S^3$ as the standard solid torus, then $L$ becomes a link inside the solid torus;
  \item    Add a meridian circle of the solid torus to $L$ to obtain a link $L'$ ;
  \item    Add a small ``earring'' to the meridian circle, the resulting link is $\widetilde{L}$. Also pick a small arc $u$ joining the earring and the meridian.
\end{enumerate}

\begin{figure}
\centering
\begin{tikzpicture}
\tikzset{
    partial ellipse/.style args={#1:#2:#3}{
        insert path={+ (#1:#3) arc (#1:#2:#3)}
    }
}
\draw[thick] (0,0) ellipse (3 and 2);  \draw[thick] (0,0) ellipse (2.65 and 1.7);  \draw[dashed] (0,0) ellipse (1.2 and 0.8);
\draw[thick] (-1.4,-0.4) rectangle (-2.2, 0.4); \node at (-1.8,0) {$L$};
\draw[thick] (0,0) [partial ellipse=-157.5:157.5: 1.62cm and 1.08cm];
\draw[thick] (0,0) [partial ellipse=-164:164: 2.22cm and 1.48cm];
\draw[thick] (0,0) [partial ellipse=-163:163: 2.07cm and 1.38cm]; \node[above] at (-1.67,0.4) {...};

\draw[red]  (-2.65,0) to (-3,0);
\draw[dashed] (-2.82,0) to (-3.2,-0.3); \node[below] at (-3.25,-0.25) {$u$};
\draw[dashed] (-3,0.1) to (-3.5, 0.4); \draw[dashed] (-2.65,0.1) to (-3.3, 0.4);  \node[above] at (-3.5,0.35) {$\mathcal{K}_2$};

\end{tikzpicture}
\caption{Link $L\sqcup\mathcal{K}_2$ in $S^1\times S^2$}\label{LK}
\end{figure}
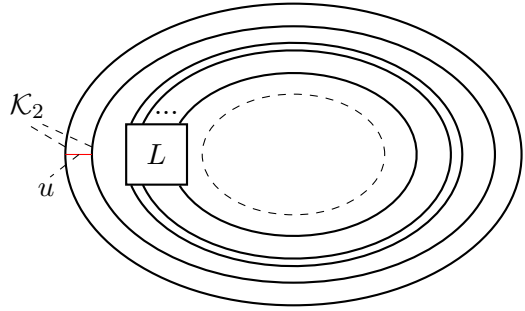

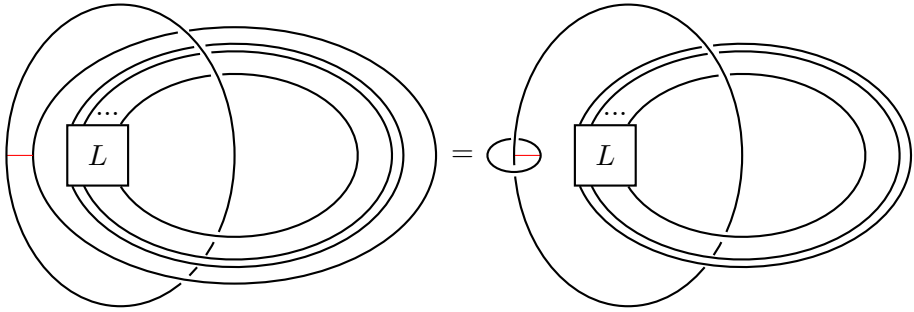
\begin{figure}
\centering
\begin{minipage}{.5\textwidth}
\begin{tikzpicture}
\tikzset{
    partial ellipse/.style args={#1:#2:#3}{
        insert path={+ (#1:#3) arc (#1:#2:#3)}
    }
}
\draw[thick,dash pattern=on 0.6cm off 0.15cm on 100cm] (0,0) [partial ellipse=90:450:2.65cm and 1.7cm];
\draw[thick] (-1.4,-0.4) rectangle (-2.2, 0.4); \node at (-1.8,0) {$L$};
\draw[thick,dash pattern=on 6.15cm off 0.15cm] (0,0) [partial ellipse=-157.5:157.5: 1.62cm and 1.08cm];
\draw[thick,dash pattern=on 8.77cm off 0.15cm] (0,0) [partial ellipse=-164:164: 2.22cm and 1.48cm];
\draw[thick,dash pattern=on 8.1cm off 0.15cm] (0,0) [partial ellipse=-163:163: 2.07cm and 1.38cm]; \node[above] at (-1.67,0.4) {...};

\draw[red]  (-2.65,0) to (-3,0);

\draw[thick,dash pattern=on 0.87cm off 0.15cm on 0.15cm off 0.2cm on 0.2cm off 0.15cm on 100cm] (-1.5,0) [partial ellipse=-90:270: 1.5cm and 2cm];
\node at (3,0) {$=$};
\end{tikzpicture}
\end{minipage}%
\begin{minipage}{.5\textwidth}
\begin{tikzpicture}
\tikzset{
    partial ellipse/.style args={#1:#2:#3}{
        insert path={+ (#1:#3) arc (#1:#2:#3)}
    }
}
\draw[thick] (-1.4,-0.4) rectangle (-2.2, 0.4); \node at (-1.8,0) {$L$};
\draw[thick,dash pattern=on 6.15cm off 0.15cm] (0,0) [partial ellipse=-157.5:157.5: 1.62cm and 1.08cm];
\draw[thick,dash pattern=on 8.77cm off 0.15cm] (0,0) [partial ellipse=-164:164: 2.22cm and 1.48cm];
\draw[thick,dash pattern=on 8.1cm off 0.15cm] (0,0) [partial ellipse=-163:163: 2.07cm and 1.38cm]; \node[above] at (-1.67,0.4) {...};

\draw[red]  (-2.65,0) to (-3,0);

\draw[thick,dash pattern=on 1.17cm  off 0.2cm on 0.2cm off 0.15cm on 6.7cm off 0.15cm on 100cm] (-1.5,0) 
[partial ellipse=-90:270: 1.5cm and 2cm];

\draw[thick, dash pattern=on 1.3cm off 0.1cm] (-3,0) [partial ellipse=-180:180: 0.35cm and 0.22cm];
\end{tikzpicture}
\end{minipage}
\caption{Link $\widetilde{L}$ in $S^3$}\label{Ltilde}
\end{figure}

If $L$ is included in a 3-ball $B\subset S^1\times S^2$, then it is clear that $(\widetilde{L},u)$ is the same as the disjoint union of
$L$  and $(H,w)$. This is the link $L^\sharp$ in \cite{KM:Kh-unknot}. Therefore we have
\begin{PR}\label{AHI-Isharp}
If a link $L$ is included in a 3-ball $B\subset S^1\times S^2$, then
\begin{equation*}
  \AHI(L)\cong \II^\sharp(L)
\end{equation*}
where $\II^\sharp(L)$ is the unreduced singular instanton Floer homology for $L$ (viewed as a link in $S^3$) defined in \cite{KM:Kh-unknot}.
Moreover,  this isomorphism respects the functoriality of both sides.
\end{PR}
For a general link $L$ in $S^1\times D$, the meridian circle in Step (2) is linked with $L$ non-trivially. Pick a basepoint on the meridian circle, then we have
\begin{equation}\label{AHI-Inatural}
  \AHI(L)\cong \II^\natural (L')
\end{equation}
where $\II^\natural$ is the reduced singular instanton Floer homology for $L'$ defined in \cite{KM:Kh-unknot}.

Let $K$ be a knot in $S^3$ and pick a base point $k_0$ on $K$. Remove a neighborhood of $k_0$ we can obtain a vertical balanced admissible tangle
$T(K)$ in $I\times D$. Pick a base point $p_0\in \mathcal{K}_1\subset S^1\times D$ and take  the connected sum
\begin{equation*}
(S^1\times D, \hat{K}):=  (S^3,{K})\# (S^1\times D, \mathcal{K}_1)
\end{equation*}
we obtain a new knot $\hat{K}$ in $S^1\times D$. Then we have
\begin{equation}\label{AAAA}
  \AHI(\hat{K},\mathbb{C})=\AHI(\hat{K},1)\oplus \AHI(\hat{K},-1)\cong \THI(T(K))\oplus \THI(T(K))
\end{equation}
by Proposition \ref{eigen-range}.
Apply Theorem \ref{Sexcision} to
\begin{equation*}
  (S^1\times S^2, \hat{K}\sqcup \mathcal{K}_2,u)\sqcup (S^1\times S^2, \mathcal{K}_3,u)
\end{equation*}
we obtain
\begin{equation}\label{BBBB}
  \THI(T(K))\cong \II(S^1\times S^2, \hat{K}\sqcup \mathcal{K}_2, \emptyset|R)
\end{equation}
where $R$ is a $S^2$-slice. The group 
$\RI(K):= \II(S^1\times S^2, \hat{K}\sqcup \mathcal{K}_2, \emptyset)$ is the \emph{reduced framed instanton homology} defined in
\cite{KM:YAFT}*{Section 4.3}. By \eqref{AAAA}, \eqref{BBBB} and Proposition \ref{specbound}, we have
\begin{equation*}
  \RI(K;\mathbb{C})\cong\II(S^1\times S^2, \hat{K}\sqcup \mathcal{K}_2, \emptyset|R)^{\oplus 2}\cong \AHI(\hat{K};\mathbb{C})
\end{equation*}
Consider
\begin{equation*}
  (S^1\times S^2, \hat{K}\sqcup \mathcal{K}_1,\omega)\sqcup (S^1\times S^2, \mathcal{K}_3, u)
\end{equation*}
where $\omega$ is an arc joining $\hat{K}$ and $\mathcal{K}_1$. Do excision along two tori $\partial N(\mathcal{K}_1)$ (one in each of the two admissible triples),
we obtain
\begin{equation*}
   (S^1\times S^2, \hat{K}\sqcup \mathcal{K}_2,\omega) \sqcup  (S^1\times S^2,  \mathcal{K}_2,u)
\end{equation*}
By Theorem \ref{Texcision} we have
\begin{equation*}
  \II(S^1\times S^2, \hat{K}\sqcup \mathcal{K}_1,\omega)^{\oplus 2}\cong \II (S^1\times S^2, \hat{K}\sqcup \mathcal{K}_2,\omega)
\end{equation*}
Using the same excision trick in \eqref{BBBB}, we can get rid of the arc $\omega$ on the right hand side of the above isomorphism:
\begin{equation*}
  \II (S^1\times S^2, \hat{K}\sqcup \mathcal{K}_2,\omega;\mathbb{C})\cong \II (S^1\times S^2, \hat{K}\sqcup \mathcal{K}_2,\emptyset;\mathbb{C})=\RI(K;\mathbb{C})
\end{equation*}
The same isomorphism with $\mathbb{Z}$-coefficients can be obtained if one use the techniques in \cite{KM:Kh-unknot}*{Theorem 5.6} to obtain an excision theorem
with $\mathbb{Z}$-coefficients. Now the same excision argument used to obtain \eqref{AHI-I} shows that
\begin{equation*}
  \II(S^1\times S^2, \hat{K}\sqcup \mathcal{K}_1,\omega)\cong \II(S^3, K^\natural)\cong \II^\natural (K)
\end{equation*}
In summary, we have
\begin{PR}\label{A-T-I}
Let $K$ be a knot in $S^3$ and $\hat{K}, T(K)$ be given as above. Then
\begin{equation*}
  \AHI(\hat{K},\mathbb{C})\cong \RI(K;\mathbb{C})\cong \II^\natural(K;\mathbb{C})^{\oplus 2}
\end{equation*}
and
\begin{equation*}
  \THI(T(K))\cong \II^\natural(K;\mathbb{C})
\end{equation*}
\end{PR}

We have a relative $\mathbb{Z}/4$ grading on $\AHI(L)$ where $L$ is a link in $S^1\times D$.
By \cite{KM:Kh-unknot}*{Section 4.5}, there is an absolute $\mathbb{Z}/4$ grading on
$\II^\natural$. Using \eqref{AHI-Inatural}, we obtain an absolute $\mathbb{Z}/4$ grading on $\AHI$. 
The operator $\muu(R)$ is of degree $2$, so
the absolute $\mathbb{Z}/4$ grading descends to an absolute $\mathbb{Z}/2$ grading on $\AHI(L,i)$.

\subsection{Unoriented Skein Exact Triangle}\label{skein-triangle}
Let $L_2$, $L_1$ and $L_0$ be three links in $S^1\times D$ which only differ in a 3-ball by unoriented skein moves, as shown in Figure \ref{L210}.
The two links $L_1$ and $L_0$ are called the 1-smoothing and 0-smoothing of $L$ at the given crossing.
The link $L_2'$ differs from $L_2$ by a change of crossing. We can also obtain $L_2'$ by rotating $L_2$ by a quarter-turn.
\begin{figure}
\centering
\begin{tikzpicture}
\draw[thick] (1,-1) to (-1,1); \draw[thick,dash pattern=on 1.3cm off 0.25cm] (1,1) to (-1,-1);  \node[below] at (0,-1.45) {$L_2$};
\draw[dashed] (0,0) circle [radius=1.414];

\draw[thick] (2,1)  to [out=315,in=90]  (2.7,0) to [out=270,in=45]    (2,-1);  \node[below] at (3,-1.45) {$L_1$};
\draw[thick] (4,1)  to [out=225,in=90]  (3.3,0) to [out=270,in=135]   (4,-1);
\draw[dashed] (3,0) circle [radius=1.414];

\draw[thick] (5,1)  to [out=315,in=180]  (6,0.3) to [out=0,in=225]   (7,1);
\draw[thick] (5,-1)  to [out=45,in=180]  (6,-0.3) to [out=0,in=135]   (7,-1);  \node[below] at (6,-1.45) {$L_0$};
\draw[dashed] (6,0) circle [radius=1.414];

\draw[thick,dash pattern=on 1.3cm off 0.25cm] (10,-1) to (8,1); \draw[thick] (10,1) to (8,-1);  \node[below] at (9,-1.45) {$L_2'$};
\draw[dashed] (9,0) circle [radius=1.414];

\end{tikzpicture}
\caption{}\label{L210}
\end{figure}

By Theorem 6.8 (and the discussion below it) in \cite{KM:Kh-unknot}, we have a cyclic exact sequence
\begin{equation*}
  \cdots \to \AHI(L_2)\to \AHI(L_1)\to \AHI(L_0) \to \AHI(L_2) \to \cdots
\end{equation*}
The maps in the exact sequence are induced by cobordisms $(W, S_{i+1,i})$ where $W$ is the product cobordism $I\times S^1\times S^2$ and $S_{i+1,i}\subset W$
is the  cobordism from $L_{i+1}$ to $L_{i}$
Therefore these maps
(after complexification) commute with $\muu(R)$ where $R$ is again the $S^2$-slice in $S^1\times S^2$. Take the generalized eigenspaces, we have the
exact sequence
\begin{equation*}
  \cdots \to \AHI(L_2,i)\to \AHI(L_1,i)\to \AHI(L_0,i) \to \AHI(L_2,i) \to \cdots
\end{equation*}
Similarly, we have the exact sequence
\begin{equation*}
  \cdots \to \AHI(L_2',i)\to \AHI(L_0,i)\to \AHI(L_1,i) \to \AHI(L_2',i) \to \cdots
\end{equation*}
From those exact sequences, we have
\begin{equation*}
  \dim \AHI(L_2,i)=\dim \AHI(L_1,i)+\dim \AHI(L_0,i)=\dim \AHI(L_2',i) ~(\text{mod}~2)
\end{equation*}
In summary, the parity of the dimension of the annular instanton Floer homology of a link at a fixed degree is invariant under crossing-change.
\begin{PR}\label{odd-AHI}
Let $L$ be a link in $S^1\times D$. If $\dim\AHI(L,i)$ is odd for some $i$, then $L$ contains no null-homologous component.
\end{PR}
\begin{proof}
Suppose $L$ contains a null-homologous component $K$. After changing crossing, we can make $K$ into an unknot $U_1$  and split it from the other components of $L$. This
process does not change the parity of $\AHI$. Denote the new link by $\tilde{L}$, then we have $\tilde L = U_1\sqcup L'$. By  \eqref{unlink-AHI}, we know
$\AHI(U_1,\mathbb{C})$ is supported at degree $0$. Therefore by Proposition \ref{link-sum}, we have
\begin{equation*}
  \AHI(\tilde L,i)=\AHI(L',i)\otimes \AHI(U_1,0)=\AHI(L',i)\otimes \mathbb{C}^2
\end{equation*}
is even-dimensional. Since $\dim \AHI(L,i)$ and $\dim \AHI(L',i)$ have the same parity, we obtain a contradiction.
\end{proof}
Using \eqref{THI=AHI}, we can also obtain the following.
\begin{PR}
Let $T\subset I\times D$ be a balanced admissible tangle. Then $T$ is vertical if and only if $\THI(T)$ is odd-dimensional.
Suppose $|T\cap D^+|$ is odd, then $T$ is vertical if and only if $\THIo(T)$ is odd-dimensional.
\end{PR} 
\begin{proof}
Suppose $|T\cap D^+|=m$. Then we have
\begin{equation*}
  \THI(T)\cong \AHI(\hat{T},m)
\end{equation*}
where $\hat{T}$ is the link obtained by closing $T$. So we know the parity of $\dim \THI(T)$ is invariant under crossing-change.

If $T$ is vertical, then $T$ could be turned into a braid by crossing-change, which has 1-dimensional instanton Floer homology. Hence $\THI(T)$ is also odd-dimensional.

Now suppose $\THI(T)$ is odd-dimensional. A closed component of $T$ becomes a null-homologous component of $\hat{T}$, which makes $\AHI(\hat{T},m)$ even-dimensional by
Proposition \ref{odd-AHI}. Therefore $T$ has no closed component. Assume there is a component $C$ of $T$ with both end points contained in
$D^+$. Move $C$ by isotopy and crossing-change, we can turn $C$ into an unknotted arc included in $(1-\epsilon, 1]\times D$ for arbitrarily small $\epsilon$.
When $\epsilon$ is small enough, the slice $\{1-\epsilon\}\times D$ intersects the new tangle $T'$ at $(m-2)$ points. So
there is a $S^2$-slice in $S^1\times S^2$ intersecting $\hat{T'}$ at $(m-2)$ points. By Proposition \ref{eigen-range} we know
\begin{equation*}
  \THI(T')\cong\AHI(\hat{T'},m)=0
\end{equation*}
which contradicts the parity assumption on $\dim \THI(T)$. Therefore we can conclude that $T$ is vertical.

The proof of the second part of the proposition is similar. 
\end{proof}
Now we can strengthen Corollary \ref{braid-detection} and Theorem \ref{odd-braid-detection} by removing the verticality assumption on $T$.
\begin{COR}
Let $T\subset I\times D$ be a balanced admissible tangle. Then $\THI(T)=\mathbb{C}$ if and only if  $T$ is isotopic to a braid.
Suppose $|T\cap D^+|$ is odd, then $\THIo(T)=\mathbb{C}$ if and only if  $T$ is isotopic to a braid.
\end{COR}

\subsection{A non-vanishing result}
All the instanton Floer homologies in this subsection are taken with complex coefficients.
Let $L\subset S^1\times D$ be an oriented link with all the components null-homologous. Pick a $D^2$-slice $D$ of $S^1\times D$ which
intersects $L$ transversely. 
For a component $K$ of $L$, if two consecutive intersection  points in $ K\cap D$ have different signs, then we can resolve
those two intersection points by doing surgery on $D$ along the arc connecting the two points in $K$. This is the same as puncturing
$D$ at the two intersection points and adding a handle $S^1\times \text{arc}$ to $D$. This operation increases the genus of $D$ by 1. 
Since $K$ is null-homologous, the number of positive intersection points equals the number of negative intersection points. 
So we can resolve all the intersection points in $L\cap D$ and obtain an admissible surface for $L$.
Recall that an admissible surface for $L$ is  a
properly embedded connected orientable surface $\Sigma$ in $S^1\times D$ 
which satisfies $\partial \Sigma\cong S^1$, $\Sigma\cap L=\emptyset$ and $[\Sigma,\partial \Sigma]$
generates $H_2(S^1\times D, \partial(S^1\times D);\mathbb{Z})\cong \mathbb{Z}$. 
The existence of such surfaces relies on
the null-homologous assumption on $L$. 

Now assume $\Sigma\subset S^1\times D$ is an admissible surface for $L$ with \emph{minimal genus}. The complement 
$$M_0=S^1\times D\setminus N(\Sigma)\cup N(L)$$ 
is a manifold with boundary 
$$\Sigma^+\cup \Sigma^-\cup I\times S^1\cup \partial N(L)$$
where $\Sigma^{\pm}$ are two copies of $\Sigma$ and $I\times S^1$ is annulus in $S^1\times \partial D$. 
We want to make $M_0$ into a balanced sutured manifold. 
For each component $K$ of $L$, add two oppositely-oriented meridian sutures to $\partial N(K)\subset \partial M_0$. 
Add a suture $\{\pt\}\times S^1$ in the annulus $I\times S^1\subset \partial M_0$. Denote those sutures by 
$\gamma_0$, we obtain a sutured manifold $(M_0,\gamma_0)$. 

Following the notation in \cite{KM:suture}, we have the (oriented) decomposition 
\begin{equation*}
\partial M_0= R_+(\gamma_0) \cup A(\gamma_0) \cup (-R_-(\gamma_0))
\end{equation*} 
where $A(\gamma_0)$ is a union of annuli, which are neighborhoods of the sutures.  
Suppose $L=K_1\cup\cdots\cup K_l$ where $K_i$ is a component of $L$. Each component $K_i$ contributes two annuli 
$A_i^+\subset R_+(\gamma_0)$ and  $A_i^-\subset R_-(\gamma_0)$ which lie in the complement of 
the two meridian sutures in $\partial N(K_i)$. We have
\begin{equation*}
R_\pm(\gamma_0)=\Sigma_\pm \cup \bigcup_i A_i^\pm
\end{equation*} 
\begin{PR}\label{M0-taut}
Suppose there is no 3-ball in $S^1\times D$ which contains some components 
of $L$ and is disjoint from the other components. Then 
the sutured manifold $(M_0,\gamma_0)$ constructed as above is taut. 
\end{PR}     
\begin{proof}

According to the definition in \cite{G:Sut-1}, to show $(M_0,\gamma_0)$ is taut we need to check three things:
\begin{itemize}
  \item The manifold $M_0$ is irreducible.
  \item The surface $R^{\pm}(\gamma_0)$ minimizes the Thurston norm of the class 
        $$[R^{\pm},\partial R^{\pm}] \in H_2 (M_0,A(\gamma_0))$$
  \item The surface $R^{\pm}(\gamma_0)$ is incompressible.
\end{itemize}

If  $M_0$ is not reducible, then there is a 2-sphere $S\subset M_0$ which does not
bound a 3-ball in $M_0$. Since $M_0\subset S^1\times D$ and $S^1\times D$
is irreducible, $S$ bounds a 3-ball $B$ in $S^1\times D$.
From the construction of $M_0$ we see that $B$ contains some components
of $L$, which contradicts the assumption.

Given a properly embedded surface $(F,\partial F)\subset (M_0,A(\gamma_0))$, the norm is defined by
\begin{equation*}
x(F):=\sum_i\max \{0, -\chi(F_i)\}
\end{equation*}
where $\{F_i\}$ are the connected components of $F$. Given a homology class in $H_2(M_0,A(\gamma_0))$, its Thurston norm is
defined as
\begin{equation*}
x(\alpha):=\min_{[F]=\alpha} x(F)
\end{equation*}
It is clear that $x(R_\pm)=x(\Sigma)$ since all the other components of $R_\pm$ are of norm 0. 
Suppose there is  
properly embedded surface $F$ whose  norm is smaller than $\Sigma$ 
and $(S,\partial S)\subset (M_0, A(\gamma_0))$ represents the same
relative homology class as $(R_{\pm},\partial R_{\pm})$ in $H_2 (M_0, A(\gamma_0))$. So $\partial F$ represents the same
homology class as $\partial R^{\pm}$ in $H_1(A(\gamma_0))$. 
We attach an annulus to the inner-most pair of oppositely-oriented circles of $\partial F$
and ``press'' this annulus 
into the interior
of $M_0$. 
We can keep doing this until there is only one circle left. Finally we obtain a new surface $\tilde{F}$ with only one boundary component
in each annulus in $A(\gamma_0)$.

If $F$ has a disk component, then the boundary of this disk must lie in $I\times S^1\subset S^1\times \partial D$ since
the meridian sutures are homologically non-trivial in $M_0$. If we glue $\Sigma_+$ and $\Sigma_-$, then this disk 
is a properly embedded surface in $S^1\times D$ which is disjoint from $L$. It has genus 0. By assumption, the genus of $\Sigma$
must also be $0$. Hence $x(R_{\pm})=x(\Sigma)=0$ and the norm of $F$ cannot be smaller than $x(R_\pm)$.

 If $F$ has no disk component, then $F$ and $\tilde{F}$ have the same norm. Throw away any sphere component in
$\tilde{F}$ and identify $\Sigma_+$ and $\Sigma_-$ in $M_0$, we obtain a surface $F'$ in $S^1\times D \setminus N(L)$ 
whose norm is smaller than $\Sigma$ and $\partial F$ consist of the meridian sutures (isotoping $F'$ if necessary) and
a circle in  $S^1\times \partial D$ which is null-homologous in $S^1\times D$. Now we pick annuli in $N(K_i)$  joining the
meridian sutures and add those annuli to $F'$ so that we obtain a new surface $F''$ with only one boundary component.
Since $F$ has no disk component, so does $F'$. Therefore $F''$ has the same norm as $F'$, which is smaller than $x(\Sigma)$.
The component of $F''$ with boundary is an admissible surface in $S^1\times D$. Its genus
is smaller than the genus $\Sigma$, which contradicts the assumption on $\Sigma$. We have $R_\pm$ is norm-minimizing. 

If there is a compressing disk $H$ for $R_\pm$, then we compress $R_\pm$ along $H$ to obtain a surface $R_\pm'$ with 
$\chi (R'_\pm)=\chi (R_\pm)+2$. The norm of
$R'_\pm$ is strictly smaller than the norm of $R_\pm$ unless $H$ bounds a meridian in an annulus component of $R^+$. 
Since we have shown that $R_\pm$ is norm-minimizing, we must have $H$ bounds a meridian in an annulus component of $R^+$.
However the meridians of the annulus components of $R_\pm$ are homologically non-trivial in $M_0$, 
hence we obtain a contradiction. This shows 
that $R_\pm$ is incompressible. 
\end{proof}

Again we pick an admissible surface $\Sigma\subset S^1\times D$ for $L$. 
This time $\Sigma$ is not necessarily genus-minimizing. 
We can still define a sutured manifold $(M_0,\gamma_0)$ as before. But $(M_0,\gamma_0)$ is not necessarily taut.
Recall that
we use an admissible triple $(S^1\times S^2, L\sqcup\mathcal{K}_2,u)$ to define $\AHI(L)$. This $S^1\times S^2$ is obtained by gluing
another copy of $S^1\times D$ to the original copy. We may assume $\partial \Sigma$ is contained in a $D$-slice in $S^1\times D$ hence
we can obtain a closed surface $\bar{\Sigma}:=\Sigma\cup D$ in $S^1\times S^2$. This surface $\bar{\Sigma}$ intersects $\mathcal{K}_2$
transversely at two points. We add a small meridian circle for each component of $L$ and add 
a small arc joining the meridian and the component of $L$. Denote the union of these meridians by $m_L$ and the union of these arcs by $u_L$.
We have a new admissible triple $(S^1\times S^2, L\cup m_L\cup \mathcal{K}_2, u+u_L)$. We use
\begin{equation*}
\II(S^1\times S^2, L\cup m_L\cup \mathcal{K}_2, u+u_L|\bar{\Sigma})
\end{equation*} 
to denote the generalized eigenspace of $\muu(\bar{\Sigma})$ with eigenvalue $2g(\bar{\Sigma})$. Notice that $\muu(\bar{\Sigma})$
is the same as $\muu(R)$ since $\bar{\Sigma}$ is homologous to the $S^2$-slice $R$.
\begin{PR}\label{suture=singular}
We have
\begin{equation*}
\SHI (M_0,\gamma_0)\cong \II(S^1\times S^2, L\cup m_L\cup \mathcal{K}_2, u+u_L|\bar{\Sigma})
\end{equation*}
where $\SHI(M_0,\gamma_0)$ is the instanton Floer homology for sutured manifolds defined in \cite{KM:suture}.
\end{PR}
\begin{proof}
Do excision to 
$$
(S^1\times S^2, L\cup m_L\cup \mathcal{K}_2, u+u_L|\bar{\Sigma})
$$
along the two boundary tori $\partial N(\mathcal{K}_2)$ and apply Theorem \ref{Texcision}, we have
\begin{equation}\label{S2=Sigma2}
\II(S^1\times S^2, L\cup m_L\cup \mathcal{K}_2, u+u_L)\cong 
\II(S^1\times \Sigma_2, L\cup m_L, \bar{u}+u_L)
\end{equation} 
where $\Sigma_2$ is a genus 2 surface obtained by puncturing $S^2$ at two points and adding a handle and
$\bar{u}$ is a closed non-separating circle in a $\Sigma_2$-slice.  Notice that we throw away a triple
$(S^1\times S^2,\mathcal{K}_2,u)$ with 1-dimensional Floer homology in the excision. 
We will keep doing this in the argument to simplify the notation.
For a component $K_i$ of $L$ with meridian $m_i\subset m_L$, let $N(K_i)$  and $N(m_i)$ be small tubular neighborhoods which are 
disjoint from each other.
 We do excision along $\partial N(K_i)$ and $\partial N(m_i)$ using
a diffeomorphism $h:\partial N(K_i)\to \partial N(m_i)$ that maps the longitude of $K_i$  to the meridian of $m_i$. 
After doing excision to all the components, again by Theorem \ref{Texcision}, we have
\begin{equation}\label{Sigma2=YL}
\II(S^1\times \Sigma_2, L\cup m_L, \bar{u}+u_L)\cong 
\II(Y(L), \emptyset, \bar{u}+\bar{u}_L)_{\mu(x),2}
\end{equation}
where $Y(L)$ is the manifold obtained by excision and $\bar{u}_L$ is the union of closed circles obtained from $u_L$ and the excision.

On the other hand, to calculate $\SHI (M_0,\gamma_0)$, we need to close up $M_0$ first. Recall
that $A(\gamma_0)$ consists of an annulus $I\times S^1$ between $\Sigma_+$ and $\Sigma_-$ on $\partial M_0$ and 
two annuli on $\partial N(K_i)$ for each component $K_i$. Let $T_0$ be genus one surface with one boundary component and 
$T_i$ ($1\le i\le l$) be an annulus. We form product sutured manifold $I\times T_i$ ($0\le i\le l$) and attach them to
$M_0$ along sutures to obtain
\begin{equation*}
M_1=M_0\cup_{A(\gamma_0)} \bigcup_{0\le i\le l} I\times T_i
\end{equation*}
To be more precise, we identify $I\times \partial T_0$ with $I\times S^1\subset \partial M_0$ and $I\times \partial T_i$ ($1\le i\le l$)
with with the two annuli on $\partial N(K_i)$. We have
\begin{equation*}
\partial M_1= \bar{R}_+ \cup \bar{R}_-
\end{equation*}
where $\bar{R}_\pm$ consist of $\bar{\Sigma}_\pm:=\{\pm 1\}\times T\cup \Sigma_\pm$ and 
a torus $S_i^\pm$ for each component $K_i$ of $L$. 
According to the discussion in \cite{KM:suture}*{Section 5.1},  $S_i^+$ and $S_i^-$ can be viewed as 
$\partial N(K_i)$ and $\partial N(m_i)$ respectively. Now we pick an identification $\bar{R}_+\to \bar{R}_-$ which identifies 
 $\bar{\Sigma}_+$ with $\bar{\Sigma}_-$ in the obvious way and identifies $S_i^+$ with $S_i^-$
 in the same way we construct $Y(L)$ (view them as $\partial N(K_i)$ and $\partial N(m_i)$).  
 It is clear the resulting manifold is again $Y(L)$. 
 We denote  the image of $\bar{\Sigma}_\pm$ in $Y(L)$ by $\bar{\Sigma}'$. The image of $I\times T_0$ in $Y(L)$ is the product
 $S^1\times T_0$. We use $c$ to denote a circle $S^1\times \{t\}\subset S^1\times T_0$. 
 By \cite{KM:Alexander}*{Section 2.3}, we have
\begin{equation}\label{YL-suture}
\SHI(M_0,\gamma_0)\cong \II(Y(L), \emptyset, c+\bar{u}_L|\bar{\Sigma}')
\end{equation}
where 
$\II(Y(L), \emptyset, c+\bar{u}_L|\bar{\Sigma}')$ is the 
 simultaneous generalized eigenspace for the operators $\mu(\bar{\Sigma}')$, $\mu(x)$ for the pair of
eigenvalues $(2g(\bar{\Sigma}'),2)$. The argument in \cite{KM:suture}*{Section 7.4} can be used to show that 
\begin{equation}\label{suture-c=u}
 \II(Y(L), \emptyset, c+\bar{u}_L|\bar{\Sigma}')\cong \II(Y(L), \emptyset, \bar{u}+\bar{u}_L|\bar{\Sigma}')
\end{equation}
The isomrophism in \eqref{Sigma2=YL} is induced by a cobordism in which $\Sigma_2$ and $\bar{\Sigma}'$ are homologous. So we have
\begin{equation*}
\II(S^1\times \Sigma_2, L\cup m_L, \bar{u}+u_L)_{\mu(\Sigma_2),2g(\bar{\Sigma}')-2}\cong 
\II(Y(L), \emptyset, \bar{u}+\bar{u}_L|\bar{\Sigma}')
\end{equation*}
Combined with \eqref{YL-suture} and \eqref{suture-c=u}, we have
\begin{equation}\label{SHI-YL}
\SHI(M_0,\gamma_0)\cong \II(S^1\times \Sigma_2, L\cup m_L, \bar{u}+u_L)_{\mu(\Sigma_2),2g(\bar{\Sigma}')-2}
\end{equation}
The isomorphism in \eqref{S2=Sigma2} intertwines the operator $\muu(\bar{\Sigma})$ with $\mu(\Sigma_2)$. 
Since $g(\bar{\Sigma}')=g(\bar{\Sigma})+1$,  we have
\begin{equation*}
\II(S^1\times S^2, L\cup m_L\cup \mathcal{K}_2, u+u_L|\bar{\Sigma})\cong 
\II(S^1\times \Sigma_2, L\cup m_L, \bar{u}+u_L)_{\mu(\Sigma_2),2g(\bar{\Sigma}')-2}
\end{equation*} 
Together with \eqref{SHI-YL}, the proof is complete. 
\end{proof}

Our main result for this subsection is
\begin{THE}\label{AHI-non-vanish}
Suppose $L\subset S^1\times D$ is a link with all the components null-homologous and $\Sigma$ is an admissible surface with minimal
genus in $S^1\times D$. Then we have
\begin{equation*}
\AHI(L,\pm 2g(\Sigma)) \neq 0
\end{equation*} 
\end{THE}
\begin{proof}
Because of the symmetry, it suffices to show that $\AHI(L, 2g(\Sigma)) \neq 0$.
Suppose $L=L_1\cup L_2$ 
and there is a 3-ball $B$ in $S^1\times D$ such that $L_1\subset B$ and 
$L_2\cap B=\emptyset$. Then we can use Proposition \ref{link-sum} to obtain 
$$
\AHI(L)\cong \AHI(L_1)\otimes \AHI(L_2).
$$
By Proposition \ref{eigen-range} (for the case $m=0$), $\AHI(L_1)$ is supported at f-grading $0$. Then the second part
of Proposition \ref{link-sum} shows that 
$$
\AHI(L,2g(\Sigma))\cong \AHI(L_1,0)\otimes \AHI(L_2,2g(\Sigma))=\AHI(L_1)\otimes \AHI(L_2,2g(\Sigma)).
$$
By Proposition \ref{AHI-Isharp}, we have $\AHI(L_1)\cong \II^\sharp(L_1)$.
In \cite[Section 3.1]{KM-Ras}, a local system $\Gamma$ is introduced and 
the instanton Floer homology $\II^\sharp(L_1;\Gamma)$ with local coefficients 
is defined as a $\mathbb{C}[t,t^{-1}]$-module. 
By \cite[Proposition 3.1]{KM-Ras} and the discussion below it, 
we have that $\II^\sharp(L_1;\Gamma)/\text{torsion}$ is isomorphic to 
$\II^\sharp(U_l;\Gamma)$ where $l$ is the number of components of $L_1$
and $U_l$ is the $l$ component unlink. 
Since $\II^\sharp(U_l;\Gamma)$ is a free 
$\mathcal{R}=\mathbb{C}[t,t^{-1}]$-module of rank $2^l$, we have
$$
\II^\sharp(L_1;\mathbb{C})\cong
\II^\sharp(L_1;\Gamma\otimes_{\mathcal{R}}\mathcal{R}/(t-1) )\neq 0
$$ by the universal coefficient theorem. Therefore it suffices
to replace $L$ by $L_2$ in order to prove the theorem. Hence we may assume 
there is no 3-ball in $S^1\times D$ which contains some components 
of $L$ and is disjoint from the other components.

By Proposition \ref{M0-taut} and \cite{KM:suture}*{Theorem 7.12}, we have
$$
\SHI(M_0,\gamma_0)\neq 0
$$
Using Proposition \ref{suture=singular}, we have
\begin{equation}\label{earrings-neq-0}
\II(S^1\times S^2, L\cup m_L\cup \mathcal{K}_2, u+u_L|\bar{\Sigma})\neq 0
\end{equation}
Next we want to relate $\II(S^1\times S^2, L\cup m_L\cup \mathcal{K}_2, u+u_L|\bar{\Sigma})$ to $\AHI(L)$, which is
defined as
\begin{equation*}
\II(S^1\times S^2, L \cup \mathcal{K}_2, u)
\end{equation*}
This means we need to remove the ``earrings'' from $(S^1\times S^2, L\cup m_L\cup \mathcal{K}_2, u+u_L)$. 

Apply Kronheimer-Mrowka's unoriented skein exact triangle (see Section \ref{skein-triangle})
to a crossing between $K_1$ and its earring $m_1$, we obtain a cyclic exact sequence
\begin{align*}
& \II(S^1\times S^2, L\cup m_L\cup \mathcal{K}_2, u+u_L)\to 
\II(S^1\times S^2, L\cup m_{L\setminus K_1}\cup \mathcal{K}_2, u+u_{L\setminus K_1}) \\
& \to \II(S^1\times S^2, L\cup m_{L\setminus K_1}\cup \mathcal{K}_2, u+u_{L\setminus K_1})\to  \cdots
\end{align*}  
Moreover, the map in this exact sequence commutes with the action of $\muu(\bar{\Sigma})=\muu(R)$ since those maps
are induced by cobordisms in which the two copies of $\bar{\Sigma}$ on the two ends are homologous. In particular, we have
an exact sequence on generalized eigenspaces with a fixed eigenvalue. 
From this exact sequence, \eqref{earrings-neq-0} implies 
\begin{equation*}
\II(S^1\times S^2, L\cup m_{L\setminus K_1}\cup \mathcal{K}_2, u+u_{L\setminus K_1})_{\muu(\bar{\Sigma}),2g(\bar{\Sigma})} \neq 0
\end{equation*}
We can repeat this argument to remove all the ``earrings'' to obtain
\begin{equation*}
\AHI(L,2g(\Sigma))=\II(S^1\times S^2, L\cup \mathcal{K}_2, u)_{\muu(\bar{\Sigma}),2g(\bar{\Sigma})} \neq 0
\end{equation*}
\end{proof}
If $\AHI(L)$ is supported at degree $0$, then an admissible surface with minimal genus must be of genus $0$ which is just a disk.
This means $L$ is included in a three-ball $B^3\subset S^1\times D$. By Proposition \ref{AHI-Isharp}, we have
\begin{COR}
Suppose $L$ is a link in $S^1\times D$ with all the components null-homologous.
If $\AHI(L)$ is supported at degree $0$, then $L$ is included in a three-ball $B^3\subset S^1\times D$ and 
$\AHI(L)\cong \II^\sharp(L)$.
\end{COR} 

\section{The Spectral Sequence}\label{SS}
Let $L$ be a link in $A\times I$ with a projection to $A$ (more precisely, $A\times\{0\}$).
This projection gives a diagram $D$ with $c$ crossings. We also assume that the crossings are ordered.
Given any element $v$ in the cube $\{0,1\}^c\subset \mathbb{R}^c$, we can resolve
the crossings by $0$-smoothing or $1$-smoothing determined by $v$. We denote the resulting link by $L_v$. 
As a link in $A$, $L_v$ is just a collection of
trivial (null-homologous) circles and non-trivial circles.

We define a partial order on
$\{0,1\}^c$ by setting
$v\ge u$ if $v_i\ge u_i$ for all $i\in \{1,\cdots,c\}$ where $v_i,u_i$ denote the $i$-th components.
For each $v\ge u$, there is a standard cobordism $S_{vu}\subset I\times A\times \{0\}$ from $L_v$ to $L_u$. All the links $L_v$ and cobordisms $S_{vu}$ can be
oriented consistently so that $\partial S_{vu}=L_u-L_v$ by the same method used 
in \cite{KM:Kh-unknot}*{Section 8.1}: take a checkerboard coloring of the regions
of the diagram $D$ in $A$ and orient $L_v$ by the boundary orientation of the black region away from the smoothings. Now for each $S_{vu}$ the map
\begin{equation*}
  \AHI(S_{vu}):\AHI(L_v)\to \AHI(L_u)
\end{equation*}
is well-defined \emph{without} an overall sign ambiguity.

By Corollaries 6.9 and 6.10 in \cite{KM:Kh-unknot}, we have
\begin{PR}\label{KM-ss}
Given links $L_v$ as above, there is a spectral sequence whose $E_1$-page is
\begin{equation*}
  \bigoplus_{v\in \{0,1\}^c} \AHI(L_v)
\end{equation*}
and which converges to $\AHI(L)$. Moreover, the differential on the $E_1$ page is
\begin{equation}\label{d1}
  d_1=\sum_i \sum_{v-u=e_i} (-1)^{\eta(v,u)}\AHI(S_{vu})
\end{equation}
where $e_i$ is the standard $i$-th basis vector in $\mathbb{R}^c$ and  $\eta(v,u)=\sum_{j=i}^c v_j$.
\end{PR}

\subsection{The operator action on the spectral sequence}
We also want to study the $\muu(R)$ action on the spectral sequence, so in this subsection we include more details 
of Kronheimer-Mrowka's 
construction of the spectral
sequence in Proposition \ref{KM-ss}. We assume perturbations are chosen so that all the moduli spaces are regular through this subsection.
We define
\begin{eqnarray*}
  |v|_1 &=& \sum_i |v_i| \\
  |v|_\infty &=& \sup_i |v_i|
\end{eqnarray*}
for any $v\in \mathbb{R}^c$.
Given any $v\in \{0,1\}^c$, we use $C_v$ to denote the Floer chain complex (under certain perturbation) for the triple $(S^1\times S^2, L_v\sqcup \mathcal{K}_2, u)$ and
use $d_v$ to denote the differential on $C_v$. Given $v\ge u$ in $\{0,1\}^c$ , the cobordism $S_{vu}$ can be made into a surface
$S_{vu}^+$ with cylindrical ends included
in $\mathbb{R}\times S^1\times S^2$ in the standard way. The surface
$S_{vu}^+$ is a product surface away from $|v-u|_1$ four-balls where the skein moves happen.
By shifting these four-balls containing the skein moves, 
Kronheimer and Mrowka define a family of metrics parametrized by $G_{vu}\cong \mathbb{R}^{|v-u|_1}$.
There is an $\mathbb{R}$-action on $G_{vu}$ defined by the $\mathbb{R}$-translation on $\mathbb{R}\times S^1\times S^2$. The quotient $G_{vu}\slash \mathbb{R}$ is denoted by
$\breve{G}_{vu}$. The spaces
$G_{vu}$ and $\breve{G}_{vu}$ are not compact in general but can be compactified into spaces
$G_{vu}^+$ and $\breve{G}_{vu}^+$ by adding broken metrics.
Let
\begin{equation*}
  M_{vu}(\alpha,\beta)_d:=M(\mathbb{R}\times S^1\times S^2, S_{vu}^+\sqcup \mathbb{R}\times \mathcal{K}_2,\mathbb{R}\times u, G_{vu};\alpha,\beta)_d
\end{equation*}
be the $d$-dimensional moduli space of ASD trajectories on $(\mathbb{R}\times S^1\times S^2, S_{vu}^+\sqcup \mathbb{R}\times \mathcal{K}_2,\mathbb{R}\times u)$
equipped with metrics in $G_{vu}$ with limiting connection $\alpha$ on the incoming end and $\beta$ on the outgoing end. 
Here $\alpha$ and $\beta$ are generators for
$C_v$ and $C_u$ respectively.
There is an obvious map
$M_{vu}\to G_{vu}$ and the $\mathbb{R}$-action on $G_{vu}$ can be lifted on $M_{vu}(\alpha, \beta)_d$. We denote the quotient by
\begin{equation*}
  \breve{M}_{vu}(\alpha,\beta)_{d-1}:=M_{vu}(\alpha, \beta)_d/\mathbb{R}
\end{equation*}
Both $M_{vu}(\alpha,\beta)_d$ and $\breve{M}_{vu}(\alpha,\beta)_{d-1}$ can be partially compactified by adding broken trajectories lying over broken metrics in
$\partial G_{vu}^+$ and $\partial\breve{G}_{vu}^+$. We denote these partial compactications by $M^+_{vu}(\alpha,\beta)$ and $\breve{M}^+_{vu}(\alpha,\beta)$ respectively.
These are only \emph{partial} compactifications because of the possible appearance of bubbles. If the dimension of these spaces are smaller than $4$ then they are compact
since no bubble could appear.

A group homomorphism
\begin{equation*}
  \breve{m}_{vu}: C_v\to C_u
\end{equation*}
can be defined by counting (signed) points in the 0-dimensional moduli space:
\begin{equation*}
  \breve{m}_{vu}(\alpha):=\sum_\beta \# \breve{M}_{vu}(\alpha,\beta)_0 \cdot \beta
\end{equation*}
where $\beta$ runs through all the generators for $C_u$. In the case $v=u$, 
$\breve{m}_{vv}$ is just the Floer differential.
Notice that the definition of $\breve{m}_{vu}$ depends on a choice of orientation of the moduli spaces and
different choices will define maps differing by an overall sign \cite{KM:Kh-unknot}*{Section 4.4}.

The boundary points of $\breve{M}_{wu}(\alpha,\beta)_1$ are broken trajectories and the signed count of the boundary points is $0$ as the boundary of an oriented 1-manifold.
A proper choice of orientations for all the moduli spaces is made in \cite{KM:Kh-unknot}. Under this choice of orientation,
$\#\partial \breve{M}_{wu}(\alpha,\beta)_1=0$ implies
\begin{PR}[{{\cite{KM:Kh-unknot}*{Lemma 6.5}}}]
Given $w\ge u$ in $\{0,1\}^c$, we have
\begin{equation}\label{mvu=}
  \sum_{w\ge v \ge u} (-1)^{|v-u|_1(|w-v|_1-1)+1}\breve{m}_{vu}\circ \breve{m}_{wv}=0
\end{equation}
\end{PR}
Now define
\begin{equation*}
  f_{vu}:C_v\to C_u
\end{equation*}
by
\begin{equation}\label{fvu}
  f_{vu}:=(-1)^{\frac{1}{2}|v-u|_1(|v-u|_1-1)+\sum v_i  }\breve{m}_{vu}
\end{equation}
Using \eqref{mvu=}, it is easy to check that
\begin{equation}\label{fvu=}
  \sum_v f_{vu} f_{wv}=0
\end{equation}
where $w\ge v\ge u$ in $\{0,1\}^c$. Therefore
\begin{equation*}
  (\mathbf{C},\mathbf{F}):=(\bigoplus_{v\in \{0,1\}^c}C_v, \sum_{v\ge u} f_{vu})
\end{equation*}
is a chain complex.
\begin{THE}[{{\cite{KM:Kh-unknot}*{Theorem 6.8}}}]\label{cube}
There is a quasi-isomorphism
\begin{equation*}
  (C(L),d_L)\to (\mathbf{C},\mathbf{F})
\end{equation*}
where $(C(L),d_L)$ is the Floer chain complex used to define $\AHI(L)$.
\end{THE}
Proposition \ref{KM-ss} follows from the above theorem by filtering the cube $\mathbf{C}$ by the sum of coordinates (and choosing the signs properly).

Besides the differential $\mathbf{F}$, we want to define another map on $\mathbf{C}$.
We move to \emph{complex coefficients} in the following discussion. Let $R$ be a sphere slice in $S^1\times S^2$ disjoint with all the three-balls
where the skein moves of $L$ happen. Let $\nu(R)\subset S^1\times S^2$ be a neighborhood of $R$ that contains $H$
and is  disjoint with all the skein moves. 
The cohomology class used to define $\muu(R)$ can be
represented as a linear combination $\sum_l a_l V_l$ where $V_l$ are divisors in $B^\ast((-1,1)\times \nu(R))$ (the space of irreducible 
connections modulo the gauge group) and $a_l\in \mathbb{Q}$. 
By abuse of notation
we denote the linear combination by $V_R$ and any intersection with $V_R$ appearing later should be understood as the linear combination of intersections with $V_l$.
Given $v\ge u$ as before, we define a map
\begin{equation*}
  r_{vu}:C_v\to C_u
\end{equation*}
by
\begin{equation}\label{rvu}
 r_{vu}(\alpha):= \sum_\beta \#  ({M}_{vu}(\alpha,\beta)_2 \cap V_R) \cdot \beta= \sum_\beta \sum_l a_l\#  ({M}_{vu}(\alpha,\beta)_2 \cap V_l) \cdot \beta
\end{equation}
where $({M}_{vu}(\alpha,\beta)_2 \cap V_l)$ should be understood as pulling back the divisor $V_l$ by $r: M_{vu}\to B^\ast((-1,1)\times \nu(R))$. We assume the divisors
$V_l$ are generic so that all the intersections are regular.
In particular, $r_{vv}:C_v\to C_v$ induces the operator $\muu(R):\AHI(L_v)\to \AHI(L_v)$ 
in homology.

The boundary of the 1-dimensional space $({M}_{wu}^+(\alpha,\beta)_3 \cap V_R)$ consists of broken trajectories of the following two types:
\begin{itemize}
  \item An element in   $ ({M}_{wv}(\alpha,\gamma)_2 \cap V_R)\times  \breve{M}_{vu}(\gamma,\beta)_0$.
  \item An element in   $\breve{M}_{wv}(\alpha, \gamma)_0 \times  (M_{vu}(\gamma, \beta)_2\cap V_R)$.
\end{itemize}
Similar to Proposition \ref{mvu=}, $\# \partial({M}_{wu}^+(\alpha,\beta)_3 \cap V_R) =0$ implies the following.
\begin{PR}
Given $w\ge u$ in $\{0,1\}^c$, we have
\begin{equation}\label{rvu=}
  \sum_{w\ge v \ge u} (-1)^{|v-u|_1(|w-v|_1-1)+1}(r_{vu}\circ \breve{m}_{wv}-\breve{m}_{vu}\circ r_{vu})   =0
\end{equation}
\end{PR}
Now we define
\begin{equation*}
  \mathbf{R}:\mathbf{C}\to \mathbf{C}
\end{equation*}
by
\begin{equation*}
  \mathbf{R}:=\sum_{v\ge u}(-1)^{\frac{1}{2}|v-u|_1(|v-u|_1-1)+\sum v_i  }r_{vu}
\end{equation*}
\eqref{rvu=} implies
\begin{PR}
$\mathbf{R}$ is a chain map on $(\mathbf{C},\mathbf{F})$: i.e. $\mathbf{RF}-\mathbf{FR}=0$.
\end{PR}

A similar $\mathbf{R}$-operator can be defined on $(C_L, d_L)$. Since there is  a unique metric in this case, the operator is nothing but $\muu(R):C_L\to C_L$ which
induces $\muu(R): \AHI(L)\to \AHI(K)$ in homology. Next we want to show that the quasi-isomorphism in Theorem \ref{cube} intertwines the $\mathbf{R}$-operators.
We need to review the proof of Theorem \ref{cube}. First of all, we want to extend the notation slightly.
Given $v\in \{0,1,2\}^c$, we can define a link $L_v$ by resolving the crossings by $0$-smoothing, $1$-smoothing
or ``2-smoothing'' determined by $v$. Here ``2-smoothing'' means to keep the crossing without any change. In particular, $L_{2,2,\cdots,2}$ is just $L$.
Let $C_v$ be the Floer chain complex used to define $\AHI(L_v)$.

Suppose $v\ge u \in  \{0,1,2\}^c$,
a family of metrics $G_{vu}\cong \mathbb{R}^{|v-u|_1}$ in defined in \cite{KM:Kh-unknot} when $|v-u|_\infty\le 1$ or
there are $v'\ge u' \in \{0,1\}^{c-1}$ such that $v=2v', u=0u'$. Using those families of metrics, we can again define $f_{vu}:C_v\to C_u$ as in \eqref{fvu} and the
same argument shows that $\eqref{fvu=}$ holds. Similarly $r_{vu}$ can be defined as in \eqref{rvu} and \eqref{rvu=} holds. Now \eqref{fvu=} and \eqref{rvu=}
imply
\begin{itemize}
  \item  The following $(c-1)$-``cube''
  \begin{equation*}
    (\mathbf{C}_2,\mathbf{F}_2):=(\bigoplus_{v'\in \{0,1\}^{c-1}}C_{2v'}, \sum_{v'\ge u'\in \{0,1\}^{c-1}}f_{2v', 2u'})
  \end{equation*}
   is a chain complex.
  \item The map
   \begin{equation*}
     \mathbf{H}: \mathbf{C}_2\to \mathbf{C}
   \end{equation*}
   defined by
   \begin{equation*}
    \mathbf{ H}:=\sum_{\substack{v'\in \{0,1\}^{c-1}\\ u\in \{0,1\}^{c}\\ 2v'\ge u }} f_{2v',u}
   \end{equation*}
   is an anti-chain map, i.e $\mathbf{F}\mathbf{H}+\mathbf{H}\mathbf{F}_2=0$.
  \item Let $$\mathbf{R}_2:\mathbf{C}_2\to \mathbf{C}_2$$ be the map defined by
  \begin{equation*}
    \mathbf{R}_2:=\sum_{v'\ge u' \in \{0,1\}^{c-1} }(-1)^{\frac{1}{2}|v'-u'|_1(|v'-u'|_1-1)+\sum v'_i  }r_{2v',2u'}
  \end{equation*}
  and
   \begin{equation*}
     \mathbf{R}': \mathbf{C}_2\to \mathbf{C}
   \end{equation*}
  be the map defined by
  \begin{equation*}
    \mathbf{ R}':=\sum_{\substack{v'\in \{0,1\}^{c-1}\\ u\in \{0,1\}^{c}\\ 2v'\ge u }}(-1)^{\frac{1}{2}|v'-u'|_1(|v'-u'|_1-1)+\sum v'_i  }  
     r_{2v',u}
   \end{equation*}
Then
\begin{equation}\label{HR}
    \mathbf{F}_2\mathbf{R}_2=\mathbf{R}_2\mathbf{F}_2~~~\text{and}~~~ 
    \mathbf{H}\mathbf{R}_2-\mathbf{R}\mathbf{H}+ \mathbf{F}\mathbf{R}'-\mathbf{R}'\mathbf{F}_2=0
  \end{equation}
 where the second equality says the map $\mathbf{R}'$ is an anti-chain homotopy between
 $\mathbf{H}\mathbf{R}_2$ and $\mathbf{R}\mathbf{H}$ 
 (see also Figure \ref{fig_chain_homotopy}). 
\end{itemize}
\begin{figure}
\[
\xymatrix{
  \mathbf{C}_2  \ar[r]^{\mathbf{H}}   \ar[d]^{\mathbf{R}_2}    &   \mathbf{C}    \ar[d]^{\mathbf{R}} \\
  \mathbf{C}_2  \ar[r]^{\mathbf{H}}       &         \mathbf{C}
}
\]
\caption{This diagram commutes in the homotopy category of chain complexes.}
\label{fig_chain_homotopy}
\end{figure}

In \cite{KM:Kh-unknot}, it is shown that
\begin{PR}
$\mathbf{H}$ induces an isomorphism in homology.
\end{PR}
Iterating the above discussion, one can define $\mathbf{C}_{2,\cdots,2}:= \bigoplus_{v'\in \{0,1\}^{c-k}}C_{2\cdots 2v'}$ and obtain a sequence of 
anti-chain maps which induce isomorphisms in homology:
\begin{equation*}
  C(L)=\mathbf{C}_{2,\cdots,2}\to \cdots \to \mathbf{C}_{2,2}\to \mathbf{C}_2 \to \mathbf{C}
\end{equation*}
Take the composition of the above quasi-isomorphisms we obtain Theorem \ref{cube}. Moreover, \eqref{HR} implies that
\begin{PR}
The isomorphism
\begin{equation*}
  H_\ast(C(L))\to H_\ast (\mathbf{C})
\end{equation*}
respects the $\mathbf{R}$-action on both sides.
\end{PR}
Since the $\mathbf{R}$-action on $\mathbf{C}$ respects the filtration used to derive the spectral sequence in Proposition \ref{KM-ss}, 
the spectral sequence can be equipped with
a $\mathbb{C}[X]$-module structure where the free variable $X$ acts by the $\mathbf{R}$-action. In particular, each generalized eigenspace 
forms a spectral sequence which is
a direct summand of the total spectral sequence. On the $E_1$-page, the $\mathbf{R}$-action is just the operator $\muu(R)$ on $\AHI$. So we have
\begin{PR}\label{AHI-ss}
Given links $L_v$ as before, for each $j\in \mathbb{Z}$ there is a spectral sequence whose $E_1$-page is
\begin{equation*}
  \bigoplus_{v\in \{0,1\}^c} \AHI(L_v,j)
\end{equation*}
and which converges to $\AHI(L,j)$. Moreover, the differential on the $E_1$ page is
\begin{equation}\label{d_1}
  d_1=\sum_i \sum_{v-u=e_i} (-1)^{\eta(v,u)}\AHI(S_{vu})
\end{equation}
where $e_i$ is the standard $i$-th basis vector in $\mathbb{R}^c$ and  $\eta(v,u)=\sum_{j=i}^c v_j$.
\end{PR}

\subsection{Differentials on the $E_1$-page}
Next we want to understand the $E_1$-page of the spectral sequence. When $v-u=e_i$, $L_v$ and $L_u$ only differ at one crossing. 
The surface $S_{vu}$ is a product cobordism
away from that crossing. It is  the union of a product cobordism and
a ``pair of pants'' cobordism embedded in $I\times S^1\times S^2$. 
We want to deal with the pair of pants part of $S_{vu}$ first. So we assume $S_{vu}$ is connected. 
There are three cases:
\begin{enumerate}[label=(\alph*)]
  \item $S_{vu}$ joins two trivial circles into one trivial circles or splits one trivial circle into two trivial circles.
  \item $S_{vu}$ joins a trivial circle and a non-trivial circle into a non-trivial circle or splits a non-trivial circle into a trivial circle and a non-trivial circle.
  \item $S_{vu}$ joins two non-trivial circles into a trivial circle or splits a trivial circle into two non-trivial circles.
\end{enumerate}

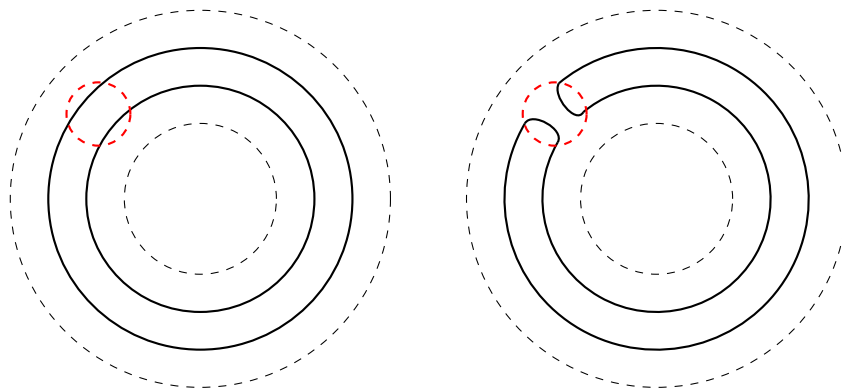
\begin{figure}
\centering
\begin{tikzpicture}
\draw[dashed] (0,0) circle [radius=1];  \draw[dashed] (0,0) circle [radius=2.5];
\draw[thick] (0,0) circle [radius=1.5];  \draw[thick] (0,0) circle [radius=2];

\draw[dashed] (6,0) circle [radius=1];  \draw[dashed] (6,0) circle [radius=2.5];
\draw[thick]  (6-1.3,0.75) arc [radius=1.5, start angle=150, end angle=490];
\draw[thick]  (6-1.732,1) arc [radius=2, start angle=150, end angle=490];

\draw[thick] (6-1.3,0.75) to [out=60, in=60]    (6-1.732,1)  ;
\draw[thick]  (6-0.964 ,1.149) to [out=240, in=240]  (6-1.285, 1.532)  ;

\draw[thick,red, dashed] (140:1.75cm) circle [radius=0.42];
\draw[thick,red, dashed] (140:1.75cm)+(6,0) circle [radius=0.42];

\end{tikzpicture}
\caption{Two non-trivial circles merge into a trivial circle in an annulus}\label{merging}
\end{figure}

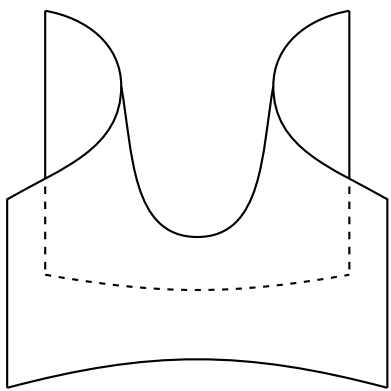
\begin{figure}
\centering

\begin{tikzpicture}
\draw[thick] (-1,0) to [out=280,in=180] (0,-2) to [out=0,in=260] (1,0); 
\draw[thick] (-2,1) to [out=350,in=90] (-1,0) to [out=270,in=30] (-2.5,-1.5) to   (-2.5,-4);

\draw[thick] (-2.5,-4) to [out=15,in=165] (2.5,-4) ;

\draw[thick] (2,1) to [out=190,in=90] (1,0) to [out=270,in=150] (2.5,-1.5) to    (2.5,-4);

\draw[dashed, thick,dash pattern=on 0.08cm off 0.1cm] (-2,-2.5) to [out=350,in=190](2,-2.5);

\draw[thick,dash pattern=on 2.23cm off 0.1cm on 0.1cm off 0.1cm on 0.1cm off 0.1cm on 0.1cm off 0.1cm on 0.1cm off 0.1cm on 0.1cm
 off 0.1cm on 0.1cm off 0.1cm on 0.1cm] (-2,1) to  (-2,-2.5);
\draw[thick,dash pattern=on 2.23cm off 0.1cm on 0.1cm off 0.1cm on 0.1cm off 0.1cm on 0.1cm off 0.1cm on 0.1cm off 0.1cm on 0.1cm
 off 0.1cm on 0.1cm off 0.1cm on 0.1cm off 0.1cm ] (2,1) to  (2,-2.5);

\end{tikzpicture}
\caption{Saddle cobordism}\label{saddle}
\end{figure}

Figure \ref{merging} shows how two non-trivial circles merge into a trivial circle in Case (c). The cobordism in Case (c) is the product cobordism outside
the dashed red circle in Figure \ref{merging} and the saddle cobordism (see Figure \ref{saddle}) inside the dashed red circle in Figure \ref{merging}.

Case (a) is the same as the situation discussed in \cite{KM:Kh-unknot}. According to the discussion in Example \ref{AHIU1K1}, we know
$\AHI(U_1)\cong \II^\sharp(U_1)\cong\mathbb{Z}^2$.
Indeed, given any link $L$ included in a three-ball in the thickened annulus, we have
\begin{equation*}
  \AHI(L)\cong \II^\sharp(L)
\end{equation*}
by Proposition \ref{AHI-Isharp}. Let $U_0$ be the empty link. Fix a generator
\begin{equation*}
  \mathbf{u}_0\in \AHI(U_0)\cong \mathbb{Z}
\end{equation*}
Let $D^+$ (resp. $D^-$) be standard disks in $I\times B^3 \subset I\times A\times I$ which give oriented cobordisms from $U_0$ to $U_1$ (resp. $U_1$ to $U_0$).
Similarly, let $\Sigma^+$  (resp. $\Sigma^-$) 
be standard punctured tori in $I\times B^3 \subset I\times A\times I$
which give oriented cobordisms from $U_0$ to $U_1$ (resp. $U_1$ to $U_0$).

We summarize some results from \cite{KM:Kh-unknot}*{Section 8} which we will need later.
\begin{PR}\label{unknot-gen}
There are generators $\mathbf{v}_+$ and $\mathbf{v}_-$ for $\AHI(U_1)\cong \mathbb{Z}^2$ characterized by
\begin{eqnarray*}
  \AHI(D^+)(\mathbf{u}_0) &=& \mathbf{v}_+ \\
  \AHI(D^-)(\mathbf{v}_-) &=& \mathbf{u}_0
\end{eqnarray*}
The degrees of two generators differ by $2$. And
\begin{eqnarray*}
  \AHI(\Sigma^+)(\mathbf{u}_0) &=& 2\mathbf{v}_- \\
  \AHI(\Sigma^-)(\mathbf{v}_+) &=& 2\mathbf{u}_0
\end{eqnarray*}
In terms of those generators, the map induced by the pair of pants cobordism from $U_2$ to $U_1$ is given by
\begin{align*}
  \mathbf{v}_+ \otimes \mathbf{v}_+ &\mapsto \mathbf{v}_+ \\
   \mathbf{v}_+ \otimes \mathbf{v}_- &\mapsto \mathbf{v}_- \\
   \mathbf{v}_- \otimes \mathbf{v}_+ &\mapsto \mathbf{v}_- \\
   \mathbf{v}_-\otimes \mathbf{v}_- &\mapsto  0
\end{align*}
The map induced by the pair of pants cobordism from $U_1$ to $U_2$ is given by
\begin{align*}
  \mathbf{v}_+ &\mapsto \mathbf{v}_+\otimes \mathbf{v}_- + \mathbf{v}_-\otimes \mathbf{v}_+  \\
  \mathbf{v}_- &\mapsto \mathbf{v}_-\otimes \mathbf{v}_-
\end{align*}
\end{PR}

Recall that
\begin{equation*}
  \AHI(\mathcal{K}_1)\cong \mathbb{Z}^2
\end{equation*}
We want to fix preferred generators so that the above isomorphism can be made into an identification. Assume that $\mathcal{K}_1$ is oriented. Let $m$ be
a meridian of $\mathcal{K}_1$ in $S^1\times S^2$ that is oriented properly: the product orientation of $S^1\times m$ coincides with the boundary orientation
from the tubular neighborhood of $\mathcal{K}_1$. From the discussion in Section \ref{tangle} we know
$\II(S^1\times S^2, \mathcal{K}_1\sqcup \mathcal{K}_2, u)$ is generated by two critical points of the Chern-Simons functional whose degree differ by $2$.
Those critical points are $SU(2)$ representations of
$$\pi_1(S^1\times S^2\setminus \mathcal{K}_1\cup \mathcal{K}_2 \cup u)$$
which map $S^1\times \{\pt\}$ to $\mathbf{i}$, a meridian of one component of $\mathcal{K}_2$ to $\mathbf{j}$ and $m$ to $\pm\mathbf{i}$.
We call the generator that maps $m$ to $\mathbf{i}$ by $\rho_+$ and the other one by $\rho_-$. To be more precise, a flat connection can only
determine a generator of the Floer chain complex up to a sign. We just pick one of the two possible generators associated to
the flat connection. 
Now we have
\begin{equation}\label{K1W1}
  \AHI(\mathcal{K}_1)=\mathbb{Z}\{\rho_+,\rho_-\}
\end{equation}
Assume $\mathcal{K}_m$ is oriented.
By excision and \eqref{K1W1} we have
\begin{equation*}
  \AHI(\mathcal{K}_m)\cong \AHI(\mathcal{K}_1)^{\otimes m}= \mathbb{Z}\{\rho_+,\rho_-\}^{\otimes m}
\end{equation*}
But the isomorphism from the excision is induced by a cobordism, which has an overall sign ambiguity. Again we pick one from
the two possibilities then we obtain an identification 
\begin{equation}\label{KmWm}
   \AHI(\mathcal{K}_m) = \mathbb{Z}\{\rho_+,\rho_-\}^{\otimes m}
\end{equation}
Now assume $\mathcal{K}_m$ is oriented.
From the discussion in Section \ref{tangle} again,  
we know  $\AHI(\mathcal{K}_m)$ is generated by $2^m$ $SU(2)$ representations (modulo conjugacy) of
$$\pi_1(S^1\times S^2\setminus \mathcal{K}_m\cup \mathcal{K}_2 \cup u)$$
which map $S^1\times \{\pt\}$ to $\mathbf{i}$, a meridian of of one component of $\mathcal{K}_2$ to $\mathbf{j}$ and $m_k$ to $\pm\mathbf{i}$ where
$m_k$ is a properly oriented meridian of the $k$-th component of $\mathcal{K}_m$. 
The element 
\begin{equation*}
  \otimes_{k=1}^m \rho_k \in \mathbb{Z}\{\rho_+,\rho_-\}^{\otimes m}
\end{equation*}
where $\rho_k=\rho_\pm$ in the identification \eqref{KmWm} corresponds one of the two generators determined by the
representation that
 maps $m_k$ to $\pm \mathbf{i}$ where the sign depends on whether $\rho_k$ is $\rho_+$ or $\rho_-$. 
\begin{PR}\label{VW}
Let $V=\mathbb{Z}\{\mathbf{v}_+,\mathbf{v}_-\}$ be the group $\AHI(U_1)$ and $W=\mathbb{Z}\{\rho_+,\rho_-\}$ be the group $\AHI(\mathcal{K}_1)$.
We have isomorphisms of $\mathbb{Z}/4$ graded abelian groups
\begin{equation*}
  \Phi_{n,m}: V^{\otimes n}\otimes W^{\otimes m}\cong \AHI(U_n\sqcup \mathcal{K}_m)
\end{equation*}
for all $n$ and $m$ satisfying the following properties:
\begin{enumerate}[label=(\roman*)]
  \item $\Phi_{0,m}$ coincides with the identification \eqref{KmWm}.
  \item Let $D^+_n$ be the cobordism from $U_0$ to $U_n$ from standard disks as before. Then
       \begin{equation*}
         \AHI(D^+_n)(\mathbf{u}_0)=\Phi_{n,0}(\mathbf{v}_+\otimes \cdots \otimes \mathbf{v}_+)
       \end{equation*}
  \item Suppose $n,m > 0$. $D^+_n\sqcup I\times \mathcal{K}_m$ is a cobordism from $\mathcal{K}_m$ to $U_n\sqcup \mathcal{K}_m$. Then
        \begin{equation*}
          \AHI(D^+_n\sqcup I\times \mathcal{K}_m)(\Phi_{0,m}(\rho_1\otimes \cdots \otimes \rho_m))=\Phi_{n,m}(\mathbf{v}_+\otimes \cdots \otimes \mathbf{v}_+\otimes
          \rho_1\otimes \cdots \otimes \rho_m)
        \end{equation*}
        where $\rho_k=\rho_\pm$.
  \item The isomorphism is canonical for split cobordisms from $U_n\sqcup \mathcal{K}_{m}$ to itself.
\end{enumerate}
\end{PR}
\begin{proof}
The isomorphisms can be obtained from Proposition \ref{link-sum}. The only issue is the sign ambiguity in the isomorphisms induced by cobordisms. The sign ambiguity
is resolved by requirements (i) (ii) and (iii) (combined with Proposition \ref{unknot-gen}). (iv) follows from the naturality of the isomorphism in
Proposition \ref{link-sum}.
\end{proof}
\begin{rmk}
 When $m=0$, this is exactly Corollary 8.5 in \cite{KM:Kh-unknot}.
\end{rmk}

We want to define an auxiliary operator on $\AHI(L)$ similar to the one used in \cite{KM:Kh-unknot}*{Section 8.3}. Let $L$ be a link in $A\times I$ with a marked point
$p$ on it. Also fix an orientation of $L$ at $p$. Form a cobordism from $L$ to itself by taking the connected sum of $I\times L$ and a standard torus at $(0,p)$.
This cobordism induces a map
\begin{equation*}
  \sigma:\AHI(L)\to \AHI(L)
\end{equation*}
of degree $2$.
\begin{PR}\label{sigma-action}
For any link $L$ in $A\times I$ and any marked point $p$ on $L$, we have $\sigma^2=0$. If $L=\mathcal{K}_m$, then $\sigma=0$ for any marked point $p$ on $\mathcal{K}_m$.
\end{PR}
\begin{proof}
By the definition of $\sigma$, we need to attach two tori to the product cobordism $I\times L$ in order to calculate $\sigma^2$.
The cobodism $W$ obtained by attaching two tori to $I\times L$ can be viewed as the composition of 
\begin{itemize}
  \item Two punctured tori $\Sigma^+$,
  \item The pair of pants cobordism,
  \item The punctured (at one point) product cobordism of $L$.
\end{itemize}
By Proposition \ref{unknot-gen}, $\AHI(\Sigma^+)(\mathbf{u}_0)=\mathbf{v}_-$ 
and $\AHI(\text{pair of pants})(\mathbf{v}_-,\mathbf{v}_-)=0$.
From the functoriality of $\AHI$, we know $\AHI(W)=0$. So we have $\sigma^2=\AHI(W)=0$.

Suppose $L=\mathcal{K}_1$. We have $\AHI(L)\cong \mathbb{Z}^2$. Decompose $\AHI(L, \mathbb{C})$ into generalized eigenspaces of $\muu(R)$, we have
\begin{equation*}
  \AHI(L, \mathbb{C})=\AHI(L,1)\oplus \AHI(L,-1)
\end{equation*}
where both eigenspaces are 1-dimensional. Since $\sigma$ commutes with $\muu(R)$, it preserves the two eigenspaces. Then $\sigma^2=0$ implies $\sigma=0$
because a nilpotent
operator on an 1-dimensional space must be $0$. $\AHI(L)$ has no torsion, so we do not lose any information using complex coefficients. By excision we know
$\sigma=0$ on $\AHI(\mathcal{K}_m)$.
\end{proof}

Now we want to understand differential maps induced by cobordisms in Case (b). Let $S_b^+$ be the cobordism from $U_1\sqcup \mathcal{K}_1$ to $\mathcal{K}_1$.
It induces a map
\begin{equation*}
  \AHI(S_b^+): \AHI(U_1\sqcup \mathcal{K}_1)\to \AHI(\mathcal{K}_1)
\end{equation*}
Using Proposition \ref{VW}, we can think of
$\AHI(S_b^+)$ as a map from $V\otimes W \to W$. Similarly let $S_b^-$ be the cobordism from $\mathcal{K}_1$ to $U_1\sqcup \mathcal{K}_1$.
Then we have a map $\AHI(S_b^-): W\to V\otimes W$.
\begin{PR}\label{Sb}
Let $y$ be an arbitrary element in $W$. We have
\begin{eqnarray*}
  \AHI(S_b^+)(\mathbf{v}_+,y) &=& y \\
  \AHI(S_b^+)(\mathbf{v}_-,y) &=& 0 \\
  \AHI(S_b^-)(y)              &=& \mathbf{v}_-\otimes y \\
 \end{eqnarray*}
\end{PR}
\begin{proof}
If we glue a disk $D^+$ to $S_b^+$ along $U_1$, $S_b^+\cup D^+$ becomes a product cobordism
$I\times \mathcal{K}_1$. By Proposition \ref{unknot-gen} and the functoriality of $\AHI$, we have
\begin{equation*}
  \AHI(S_b^+)(\mathbf{v}_+,y)=\AHI(S_b^+\cup D^+)(y)=y
\end{equation*}
If we glue a punctured torus $\Sigma^+$ to  $S_b^+$ along $U_1$, then $S_b^+\cup \Sigma^+$ is the same as the connected sum of the product cobordism
$I\times \mathcal{K}_1$ and a torus. By Proposition \ref{unknot-gen}, Proposition \ref{sigma-action} and the functoriality of $\AHI$, we have
\begin{equation*}
  \AHI(S_b^+)(\mathbf{v}_-,y)=\AHI(S_b^+\cup \Sigma^+)(y)=\sigma(y)=0
\end{equation*}

Suppose
\begin{equation*}
  \AHI(S_b^-)(y)=\mathbf{v}_+\otimes y_1 + \mathbf{v}_-\otimes y_2
\end{equation*}
where $y_1,y_2\in W$. Glue a disk $D^-$ to $S_b^-$ along $U_1$, $S_b^-$ becomes the product cobordism $I\times \mathcal{K}_1$.
Using Proposition \ref{unknot-gen} and the functoriality of $\AHI$, we have
\begin{equation*}
  y_2=\AHI(S_b^-\cup D^-)(y)=y
\end{equation*}
Glue a punctured torus $\Sigma^-$ to $S_b^-$ along $U_1$, then
then $S_b^-\cup \Sigma^-$ is the same as the connected sum of the product cobordism
$I\times \mathcal{K}_1$ and a torus. By Proposition \ref{unknot-gen}, Proposition \ref{sigma-action} and the functoriality of $\AHI$,
\begin{equation*}
  y_1=\AHI(S_b^-\cup\Sigma^-)(y)=0
\end{equation*}
Therefore we have
\begin{equation*}
   \AHI(S_b^-)(y)= \mathbf{v}_-\otimes y
\end{equation*}
\end{proof}

Let $S_c^+$ and $S_c^-$ be the cobordisms in Case (c). They are cobordisms from $\mathcal{K}_2$ to $U_1$ and 
$U_1$ to $\mathcal{K}_2$ respectively. 
So they induce maps from $W\otimes W$ to $V$ and $V$ to $W\otimes W$ respectively.
\begin{PR}\label{Sc}
There is a number $\lambda=\pm1$ such that
\begin{eqnarray*}
  \AHI(S_c^+)(\rho_+,\rho_+) &=& \lambda \mathbf{v}_- \\
  \AHI(S_c^+)(\rho_-,\rho_-)  &=& -\lambda \mathbf{v}_- \\
  \AHI(S_c^+)(\rho_+,\rho_-) &=& \AHI(S_c^+)(\rho_-,\rho_+)=0 \\
  \AHI(S_c^-)(\mathbf{v}_+) &=& \lambda (\rho_+\otimes \rho_+ -\rho_-\otimes \rho_-)\\
   \AHI(S_c^-)(\mathbf{v}_-) &=& 0
\end{eqnarray*}
\end{PR}

\begin{proof}
If we glue a disk $D^-$ along $U_1$ to $S_c^+$, then $S_c^+\cup D^-$ becomes a cobordism from $\mathcal{K}_2$ to the empty set.
The link
$\mathcal{K}_2\subset A\times I=S^1\times D$ can be described as $S^1\times \{p_1,p_2\}$ where $p_1,p_2$ are two points in $D$. Let $w$ be an arc in $D$ joining
$p_1$ and $p_2$. Then the cobordism $S_c^+\cup D^-$
can be obtained by pushing $S^1\times w\subset \{-1\}\times S^1\times D$ a little bit into
the 4-dimensional cylinder $[-1,1]\times S^1\times D$.  
Alternatively we can push $w$ a little bit into $[-1,1]\times D$ 
to obtain an arc $\tilde{w}$ and consider
$$S^1\times \tilde{w}\subset S^1\times [-1,1]\times D.$$
The cobordism $S_c^+\cup D^-$ can be identified with $F(S^1\times \tilde{w})$
where
$$
F: S^1\times [-1,1]\times D {\cong} [-1,1]\times S^1\times D
$$
is the obvious diffeomorphism. Notice that
\begin{equation}\label{eq_pi1_Sc}
\pi_1([-1,1]\times S^1\times D- S_c^+\cup D^- )\cong 
\pi_1( S^1\times ([-1,1]\times D-\tilde{w} ))\cong \mathbb{Z}\times \mathbb{Z} 
\end{equation}
whose two generators are $[S^1\times \{\pt\}]$ and the meridian of $\tilde{w}$.
Since $S_c^+\cup D^-$ is an oriented cobordism,
we may assume the orientation of $S^1\times \{p_1\}$ coincides with the orientation of $S^1$ and $S^1\times \{p_2\}$ has the opposite
orientation. The index formula shows that the 0-dimensional moduli space
\begin{equation*}
  M(I\times S^1\times S^2, S_c^+\cup D^-\sqcup I\times  \mathcal{K}_2,I\times u;\rho_{+}\otimes \rho_+, \mathbf{u}_0)
\end{equation*}
consists of only flat connections. By \eqref{eq_pi1_Sc}, we see that 
there is a unique flat connection in the moduli space: 
 the flat connection whose holonomy around $S_c^+\cup D^-$ is equal to the holonomy 
 of $\rho_+$ around 
$\mathcal{K}_1$. 
Similarly, The moduli space
\begin{equation*}
  M(I\times S^1\times S^2, S_c^+\cup D^-\sqcup I\times  \mathcal{K}_2,I\times u;\rho_{-}\otimes \rho_-, \mathbf{u}_0)
\end{equation*}
also consists of a single point.
Therefore we have
\begin{eqnarray*}
  \AHI(S_c^+\cup D^-)(\rho_+,\rho_+) &=& \lambda_1 \mathbf{u}_0 \\
   \AHI(S_c^+\cup D^-)(\rho_-,\rho_-) &=& \lambda_2\mathbf{u}_0
\end{eqnarray*}
where $\lambda_1$ and $\lambda_2$ are $\pm 1$. The two numbers 
$\lambda_1$ and $\lambda_2$ are possibly different. By Proposition \ref{unknot-gen} and
the functoriality of $\AHI$, we have
\begin{eqnarray*}
  \AHI(S_c^+)(\rho_+,\rho_+) &=& \lambda_1 \mathbf{v}_- \\
   \AHI(S_c^+)(\rho_-,\rho_-) &=& \lambda_2\mathbf{v}_-
\end{eqnarray*}

For degree reason, we have
\begin{equation*}
   \AHI(S_c^+)(\rho_+,\rho_-) = c \mathbf{v}_+
\end{equation*}
where $c$ is a constant. If we glue a punctured torus $\Sigma^-$ along $U_1$ to $S_c^+$, we obtain
\begin{equation*}
   \AHI(S_c^+\cup \Sigma^-)(\rho_+,\rho_-)=c\mathbf{u}_0
\end{equation*}
by Proposition \ref{unknot-gen} as before. On the other hand, $S_c^+\cup \Sigma^-$ is the same as the connected sum of $S_c^+\cup D^-$ and a torus, so we have
\begin{equation*}
  \AHI(S_c^+\cup \Sigma^-)(\rho_+,\rho_-)=\AHI(S_c^+\cup D^-)(\sigma(\rho_+),\rho_-)=0
\end{equation*}
by Proposition \ref{sigma-action}. Hence $c=0$ and
\begin{equation*}
   \AHI(S_c^+)(\rho_+,\rho_-) = 0
\end{equation*}
By the same argument, we have
$$\AHI(S_c^+)(\rho_-,\rho_+) = 0$$

Let $R$ be a $S^2$ slice in $S^1\times S^2$. The operator $\muu(R)$ is of degree $2$ on $W=\AHI(\mathcal{K}_1)$ with eigenvalues $\pm 1$. Because of the
$\mathbb{Z}/4$-grading, we must have
\begin{equation*}
  \muu(R)(\rho_+)=\hbar \rho_-
\end{equation*}
where $\hbar =\pm 1$. Then $\rho_+ + \hbar \rho_-$ is an eigenvector with eigenvalue $1$. And
$$(\rho_+ + \hbar \rho_-)\otimes (\rho_+ + \hbar \rho_-)\in W\otimes W$$
is an eigenvector for $\muu(R)$ with eigenvalue $2$. The map $\AHI(S_c^+)$ respects the $\muu(R)$-action. Since
$\AHI(U_1)$ only has zero eigenvalues under the $\muu(R)$-action, we must have
\begin{eqnarray*}
  0 &=& \AHI(S_c^+)(\rho_+ + \hbar \rho_-,\rho_+ + \hbar \rho_-) \\
    &=& \AHI(S_c^+)(\rho_+,\rho_+)+  \AHI(S_c^+)(\rho_-,\rho_-)        \\
    &=& (\lambda_1+\lambda_2)\mathbf{v}_-
\end{eqnarray*}
Hence $\lambda:=\lambda_1=-\lambda_2$.

The cobordism $S_c^-$ is just an orientation-reversed copy of $S_c^+$. A dual argument shows
\begin{eqnarray*}
  \AHI(S_c^-)(\mathbf{v}_+) &=& \lambda_3 (\rho_+\otimes \rho_+ -\rho_-\otimes \rho_-)\\
   \AHI(S_c^-)(\mathbf{v}_-) &=& 0
\end{eqnarray*}
where $\lambda_3=\pm 1$.

Now using previous calculations and the functoriality of $\AHI$,  we have 
\begin{equation*}
\AHI(D^+\cup S_c^-\cup S_c^+\cup D^-)(\mathbf{u}_0)=2\lambda \lambda_3 \mathbf{u}_0
\end{equation*}
We use $T$ to denote the surface $D^+\cup S_c^-\cup S_c^+\cup D^-$.
In order to show $\lambda_3=\lambda$, it suffices to prove $\AHI(T)(\mathbf{u}_0)=2\mathbf{u}_0$.  

Suppose
$\AHI(T)(\mathbf{u}_0)=a\mathbf{u}_0$.
The surface $T$ is a torus embedded in 
$[-1,1]\times S^1\times S^2$. Pick a small circle $C$ on $S^2$, then $T$ is isotopic to
$\{0\}\times S^1\times C$. Next we want to calculate $a$. 
We want to reduce the calculation to \cite[Lemma 8.8]{KM:Kh-unknot}
which says the following:
Suppose $S$ is an oriented torus embedded in $[-1,1]\times S^3$ viewed as a cobordism
from the empty link $U_0$ to itself, then 
$$
\II^\sharp(S):\II^\sharp(U_0)\to \II(U_0)
$$  
is a multiplication by $2$.

The excision in Section \ref{re-inst} changes $(S^1\times S^2,\mathcal{K}_2,u)$ 
into $(S^3, H,\omega)$ where $H$ is a Hopf link and $\omega$ is an arc 
joining the two components of $H$. Using this
excision, we have
\begin{equation*}
\II(I\times S^3,I\times H \cup T',I\times \omega)(\mathbf{u}_0')=a \mathbf{u}'_0 
\end{equation*}
where $\mathbf{u}'_0$ is a generator of $\II(S^3,H,w)=\II^\sharp(U_0)$ 
and $T'\subset \{0\}\times S^3$ is 
the boundary torus of the neighborhood of a meridian around
a component of $\{0\}\times H$. 
Closing up the two ends of $S^3$, we obtain a closed triple
$$
(S^1\times S^3, S^1\times H\cup T', S^1\times \omega).
$$
According to \cite[Proposition 5.5]{KM:Kh-unknot} (and the remark below it), 
$$
\II(S^1\times S^3, S^1\times H\cup T', S^1\times \omega)=2a.
$$
By \cite[Theorem 1.1]{Kr-ob}, the singular Donaldson invariants depend only on 
the homotopy class of the embedded surface. Even though 
our situation is not exactly the same  as the situation in \cite{Kr-ob} 
where the embedded surface is connected and the orbifold 
bundle extends,  the argument still works in our case.  
We can use a homotopy to move $T'$ so that it becomes a torus $T''$ 
in $\{0\}\times S^3$ that is contained in a ball disjoint from
$S^1\times H$. 
Then we have
$$
\II(S^1\times S^3, S^1\times H\cup T'', S^1\times \omega)=2a.
$$
Cutting $S^1\times S^3$ along a $S^3$-slice and using 
\cite[Proposition 5.5]{KM:Kh-unknot} again we obtain
$$
a\mathbf{u}_0'=\II(I\times S^3,I\times H \cup T'',I\times \omega)(\mathbf{u}_0')=
\II^\sharp(T'')(\mathbf{u}_0')
$$
By \cite{KM:Kh-unknot}*{Lemma 8.8}, we have $\II^\sharp(T'')(\mathbf{u}'_0)=2\mathbf{u}'_0$.
Therefore $a=2$.
\end{proof}

Now we move to complex coefficients. Denote $V\otimes_\mathbb{Z} \mathbb{C}$ and $W\otimes_\mathbb{Z} \mathbb{C}$ by $V_c$ and $W_c$ respectively.
Define
\begin{equation}\label{w+w-}
 \mathbf{w}_+:=\frac{1}{\sqrt{2\lambda}}(\rho_+ +\hbar \rho_-),~\mathbf{w}_-:=\frac{1}{\sqrt{2\lambda}}(\rho_+ -\hbar \rho_-)
\end{equation}
where $\hbar$ is defined by
\begin{equation*}
\muu(R)(\rho_+)=\hbar\rho_-
\end{equation*}
The two elements
$\mathbf{w}_+, \mathbf{w}_-$ form a basis for $W_c$. 
Notice that this basis is not homogeneous under the $\mathbb{Z}/4$-grading so that the $\mathbb{Z}/4$-grading
descends to a $\mathbb{Z}/2$-grading.
Using this basis, we summarize Proposition \ref{Sb} and Proposition \ref{Sc} in the following.
\begin{PR}\label{vwTQFT}
In terms of the generators $\mathbf{v}_+,\mathbf{v}_-$ for $V_c$ and $\mathbf{w}_+, \mathbf{w}_+$ for $W_c$, the maps induced by  cobordisms
$S_b^+$ and $S_b^-$ are given by
\begin{eqnarray*}
  \mathbf{v}_+\otimes \mathbf{w}_+ &\mapsto& \mathbf{w}_+ \\
  \mathbf{v}_+\otimes \mathbf{w}_- &\mapsto& \mathbf{w}_- \\
  \mathbf{v}_-\otimes \mathbf{w}_+ &\mapsto& 0\\
  \mathbf{v}_-\otimes \mathbf{w}_- &\mapsto& 0\\
\end{eqnarray*}
and
\begin{eqnarray*}
   \mathbf{w}_+ &\mapsto& \mathbf{v}_-\otimes \mathbf{w}_+ \\
   \mathbf{w}_- &\mapsto& \mathbf{v}_-\otimes \mathbf{w}_-
\end{eqnarray*}
The maps induced by cobordisms $S_c^+$ and $S_c^-$ are given by
\begin{eqnarray*}
  \mathbf{w}_+\otimes \mathbf{w}_+ &\mapsto& 0 \\
  \mathbf{w}_-\otimes \mathbf{w}_- &\mapsto& 0\\
  \mathbf{w}_+\otimes \mathbf{w}_- &\mapsto& \mathbf{v}_-\\
  \mathbf{w}_-\otimes \mathbf{w}_+ &\mapsto& \mathbf{v}_-\\
\end{eqnarray*}
and
\begin{eqnarray*}
  \mathbf{v}_+ &\mapsto& \mathbf{w}_+\otimes \mathbf{w}_- + \mathbf{w}_-\otimes \mathbf{w}_+ \\
  \mathbf{v}_- &\mapsto& 0
\end{eqnarray*}
\end{PR}

\begin{PR}\label{homotopy-dep}
Let $S, S':L_1\to L_2$ be oriented cobordisms between links $L_1,L_2\subset A\times I$. If $L_i=U_{m_i}\sqcup \mathcal{K}_{l_i}$ ($i=1,2$)
and $S, S'$ are \emph{homotopic} relative to the boundary, then $\AHI(S)=\AHI(S')$.
\end{PR}
\begin{proof}
We need to use the local system introduced in \cite{KM:YAFT}. Let $\mathcal{R}$ be the ring 
$\mathbb{Z}[t,t^{-1}]$ and $\mathcal{B}$ be the configuration space of orbifold connections on 
$$(S^1\times S^2,L_i\sqcup \mathcal{K}_2,u)$$
Given any continuous function 
$$
s:\mathcal{B}\to S^1
$$
a local system of free rank-1 $\mathcal{R}$-modules on $\mathcal{B}$ is defined in  \cite{KM:YAFT}*{Section 3.9}.
Given any orbifold connection $[A]\in \mathcal{B}$, its holonomy along any component of $L_i$ lies in a $S^1$-subgroup of $SU(2)$.
In particular, we can take $s$ to be the product of holonomies along all the components of $L_i$. In this way we obtain a 
local system $\Gamma$ and the corresponding Floer homology group
$$
\AHI(L_i;\Gamma)
$$
which is the homology of a chain complex of free $\mathcal{R}$-modules generated by the  critical points of the Chern-Simons functional. 
The Floer homology $\AHI(\mathcal{K}_1;\Gamma)$ is a free $\mathcal{R}$-module of rank 2 since  
the critical points of the corresponding Chern-Simons functional consists of two points differing by an even degree. 
Similarly we have $\AHI(U_1;\Gamma)$ is a free $\mathcal{R}$-module of rank 2. 
Theorem \ref{L1+L2} still holds for
Floer homology with local coefficients, so we have $\AHI(L_i;\Gamma)$ ($i=1,2$) is a free $\mathcal{R}$-module. 

We can view $\mathbb{Z}$ as the $\mathcal{R}$-module $\mathcal{R}/(t-1)$. In order to show that
$$
\AHI(S;\mathbb{Z})=\AHI(S';\mathbb{Z})
$$
it suffices to show 
\begin{equation}\label{AHI=Gamma}
\AHI(S;\Gamma)=\AHI(S';\Gamma)
\end{equation}
by the universal coefficient theorem. Since $S$ and $S'$ are homotopic, we have
\begin{equation}\label{AHI=Gamma=t}
(t^{-1}-t)^m\AHI(S;\Gamma)=(t^{-1}-t)^m \AHI(S';\Gamma)
\end{equation}
for some $m\in \mathbb{N}$ by \cite{KM:YAFT}*{Proposition 5.2}. 

Now \eqref{AHI=Gamma} follows from \eqref{AHI=Gamma=t} since $\AHI(L_2;\Gamma)$ is a free (hence torsion-free)  $\mathcal{R}$-module.

\end{proof}

Now we are ready to prove our main result:
\begin{THE}\label{AKh-AHI}
Let $L\subset A\times I$ be a link and $\overline{L}$ be its mirror image.
For each $i\in \mathbb{Z}$, there is a spectral sequence whose $E_2$-page is the annular Khovanov homology
$\AKh(\overline{L},i;\mathbb{C})$ and which converges to  the annular instanton Floer homology $\AHI(L,i)$.
\end{THE}
\begin{proof} 
Notice that we have a grading-preserving identification of vector spaces 
$$
CKh_{\bar{v}}(\overline{L};\mathbb{C})=CKh_{v}(L;\mathbb{C})=\AHI(L_v;\mathbb{C})
$$ 
where $\bar{v}=(1,\cdots,1)-v$.
We want to compare the differential on the $E_1$-page of the spectral sequence in Proposition \ref{AHI-ss}
and the differential used to define $\AKh(\overline{L};\mathbb{C})$.
We already know the cobordism $S_{vu}$ is the union of a pair of pants cobordism and a product cobordism. Using Proposition \ref{homotopy-dep},
we may also assume $S_{vu}$ splits: the pair of pants part and the product part are included in disjoint subspace
$S^1\times D_1$ and $S^1\times D_2$ of $S^1\times D$. Since the isomorphism \eqref{link-sum} is natural with respect to split cobordisms,
so $\AHI(S_{vu})$ can be calculated using Proposition \ref{unknot-gen} and Proposition \ref{vwTQFT}.
Therefore we know the cube for the $E_1$-page of
the spectral sequence is almost the same as the cube used to define 
$\AKh(\overline{L};\mathbb{C})$ except the sign correction term in the differential. 
The differential on the $E_1$-page is 
\begin{equation*}
\sum_i\sum_{v-u=e_i}(-1)^{\sum_{i\le j\le c} v_j} \AHI(S_{vu})
\end{equation*}  
The differential on $CKh(\overline{L})$ is
\begin{equation*}
\sum_i\sum_{v-u=e_i}(-1)^{\sum_{i< j\le c} \bar{v}_j} \AHI(S_{vu})= 
\sum_i\sum_{v-u=e_i}(-1)^{c+i+1+\sum_{i\le j\le c} v_j} \AHI(S_{vu})
\end{equation*} 
where $c$ is the number of crossings of the diagram of $L$.
We have an (anti)isomorphism 
\begin{equation*}
E_1\to CKh(\overline{L})
\end{equation*} 
which is equal to $(-1)^{\sum_k v_{2k+1}}\idd$ on $\AHI(L_v)$. Therefore the $E_2$-page is isomorphic to $\AKh(\overline{L};\mathbb{C})$.
Taking the generalized eigenspaces of $\muu(R)$ on the $E_1$-page, we obtain the desired result.   
\end{proof}

\section{Applications of the Spectral Sequence}\label{app-ss}
 Recall that given a balanced admissible tangle $T\subset D\times I$, we can define a closure $\hat{T}$ 
 which is a link in $S^1\times D = A\times I$.
Assume $|T\cap D^{+}|=m$, we have
\begin{equation*}
  \TKh(T)=\AKh(\hat{T},m)
\end{equation*}
Using Isomorphism  \eqref{THI=AHI} and Theorem \ref{AKh-AHI}, we have
\begin{THE}\label{TKh-THI}
Let $T\subset D\times I$ be a balanced admissible tangle and $\overline{T}$ be its mirror image.
There is a spectral sequence whose $E_2$-page is the tangle Khovanov homology
$\TKh(\overline{T};\mathbb{C})$ and which converges to  the tangle instanton Floer homology $\THI(T)$.
\end{THE}
\begin{COR}
Let $T\subset D\times I$ be a balanced admissible tangle. Then 
$\TKh(T,\mathbb{C})\cong\mathbb{C}$ if and only if $T$ is isotopic to a braid.
\end{COR}
\begin{proof}
The ``if'' part is simple: in all the possible resolutions, only the trivial product tangle contributes to $\TKh(T;\mathbb{C})$. 
We focus on the ``only if'' part. Suppose $\TKh(T;\mathbb{C})=\mathbb{C}$, then by Theorem \ref{TKh-THI} we have
$$
\dim \THI(\overline{T}) \le 1
$$
On the other hand, the parity of the total dimension of each page in the spectral sequence does not change after taking homology, so we know
$\dim \THI(\overline{T})$ is odd. Therefore we have $\THI(\overline{T})\cong \mathbb{C}$. By Corollary \ref{braid-detection}, we conclude 
$\overline{T}$ and $T$ are isotopic to braids.
\end{proof}

Given a knot $K\subset S^3$, we can form a tangle $T(K)$ by removing a neighborhood of a point on $K$.
The tangle $T(K)$ is vertical, balanced and admissible.
Let $T^n(K)$ be the $n$-cable of $T(K)$. We have
\begin{equation*}
  \Khr_n(K):=\TKh(T^n(K))
\end{equation*}
where $\Khr_n(K)$ is the reduced $n$-colored Khovanov homology of $K$. 
\begin{THE}
Let $K$ be a knot in $S^3$ and $\overline{K}$ be its mirror image.
There is a spectral sequence whose $E_2$-page is the reduced $n$-colored Khovanov homology
$\Khr_n(\overline{K};\mathbb{C})$ and which converges to  the reduced singular instanton Floer homology $\II^\natural(K;\mathbb{C})$.         
\end{THE}
\begin{proof}
By Theorem \ref{TKh-THI}, we have a spectral sequence from $\Khr_n(\overline{K};\mathbb{C})$ to $\THI(T^n(K))$. It suffices to show that
\begin{equation*}
  \THI(T^n(K))\cong\II^\natural(K;\mathbb{C})
\end{equation*}
Recall that we defined a knot $\hat{K}$ in $S^1\times D$ in Section \ref{re-inst}. Let $\hat{K}^n$ be the $n$-cable of $\hat{K}$. Then we have
\begin{equation*}
  \THI(T^n(K))=\II(S^1\times S^2, \hat{K}^n\sqcup \mathcal{K}_2,u|R)
\end{equation*}
where $u$ is an arc joining the two components of $\mathcal{K}_2$ and $R$ is a $S^2$-slice.
Using Proposition \ref{A-T-I} and the discussion before it, we have
\begin{equation}\label{EEEEE}
 \II^\natural(K;\mathbb C)\cong \THI(T(K))=\II(S^1\times S^2, \hat{K}\sqcup \mathcal{K}_2,\omega|R)
\end{equation}
where $\omega$ is an arc joining $\hat{K}$ and one component of $\mathcal{K}_2$. Do excision to
\begin{equation*}
  (S^1\times S^2, \hat{K}\sqcup \mathcal{K}_2,\omega)\sqcup (S^1\times S^2, \mathcal{K}_{n+1},u)
\end{equation*}
along $\partial N(\hat{K})$ and $\partial N(\mathcal{K}_1)$, we obtain
\begin{equation*}
  (S^1\times S^2, \hat{K}^n\sqcup \mathcal{K}_2,\omega)\sqcup (S^1\times S^2,\mathcal{K}_2,u)
\end{equation*}
The isomorphism in Theorem \ref{Texcision} intertwines the sphere operator on the incoming end with the sphere operator on the ongoing end for the same reason
as in the proof of Proposition \ref{link-sum}, so we have
\begin{equation}\label{DDDDD}
  \II(S^1\times S^2, \hat{K}\sqcup \mathcal{K}_2,\omega|R)\cong \II(S^1\times S^2, \hat{K}^n\sqcup \mathcal{K}_2,\omega|R)
\end{equation}
If $n$ is odd, we can apply Theorem \ref{Sexcision} to
\begin{equation*}
  (S^1\times S^2, \hat{K}^n\sqcup \mathcal{K}_2,\omega)\sqcup (S^1\times S^2, \mathcal{K}_{n+2},\omega+ u)
\end{equation*}
to obtain
\begin{equation*}
  \II(S^1\times S^2, \hat{K}^n\sqcup \mathcal{K}_2,\omega|R)\cong \II(S^1\times S^2, \hat{K}^n\sqcup \mathcal{K}_2,u|R)
\end{equation*}
When $n$ is even, one can apply Theorem \ref{Texcision} and Theorem \ref{Sexcision} to show that
\begin{equation*}
  \II(S^1\times S^2, \hat{K}^n\sqcup \mathcal{K}_2,\omega|R)\cong\II(S^1\times S^2, \hat{K}^n\sqcup \mathcal{K}_3,\omega|R)\cong
  \II(S^1\times S^2, \hat{K}^n\sqcup \mathcal{K}_3,u|R)
\end{equation*}
Then apply Theorem \ref{Texcision} again to obtain
\begin{equation*}
  \II(S^1\times S^2, \hat{K}^n\sqcup \mathcal{K}_3,u|R)\cong \II(S^1\times S^2, \hat{K}^n\sqcup \mathcal{K}_2,u|R)
\end{equation*}
So we have
\begin{equation}\label{CCCCC}
   \II(S^1\times S^2, \hat{K}^n\sqcup \mathcal{K}_2,\omega|R)\cong \II(S^1\times S^2, \hat{K}^n\sqcup \mathcal{K}_2,u|R)
\end{equation}
for any $n$. Combine \eqref{EEEEE}, \eqref{DDDDD} and \eqref{CCCCC}, we have
\begin{equation*}
  \THI(T^n(K))\cong \II^\natural(K;\mathbb C)
\end{equation*}
\end{proof}

\bibliography{references}

\begin{bibdiv}
\begin{biblist}

\bib{APS}{article}{
      author={Asaeda, Marta~M.},
      author={Przytycki, J\'ozef~H.},
      author={Sikora, Adam~S.},
       title={Categorification of the {K}auffman bracket skein module of
  {$I$}-bundles over surfaces},
        date={2004},
        ISSN={1472-2747},
     journal={Algebr. Geom. Topol.},
      volume={4},
       pages={1177\ndash 1210},
         url={https://doi.org/10.2140/agt.2004.4.1177},
      review={\MR{2113902}},
}

\bib{Bloom}{article}{
      author={Bloom, Jonathan~M.},
       title={A link surgery spectral sequence in monopole {F}loer homology},
        date={2011},
        ISSN={0001-8708},
     journal={Adv. Math.},
      volume={226},
      number={4},
       pages={3216\ndash 3281},
         url={https://doi.org/10.1016/j.aim.2010.10.014},
      review={\MR{2764887}},
}

\bib{Kh-unlink}{article}{
      author={Batson, Joshua},
      author={Seed, Cotton},
       title={A link-splitting spectral sequence in {K}hovanov homology},
        date={2015},
        ISSN={0012-7094},
     journal={Duke Math. J.},
      volume={164},
      number={5},
       pages={801\ndash 841},
         url={https://doi.org/10.1215/00127094-2881374},
      review={\MR{3332892}},
}

\bib{BS}{article}{
      author={Baldwin, John~A},
      author={Sivek, Steven},
       title={Khovanov homology detects the trefoils},
        date={2018},
     journal={arXiv preprint arXiv:1801.07634},
}

\bib{DK}{book}{
      author={Donaldson, Simon},
      author={Kronheimer, Peter},
       title={The geometry of four-manifolds},
   publisher={Oxford University Press},
        date={1990},
}

\bib{Floer-surgery}{article}{
      author={Floer, Andreas},
       title={Instanton homology, surgery, and knots},
        date={1990},
     journal={Geometry of low-dimensional manifolds},
      volume={1},
       pages={97\ndash 114},
}

\bib{G:Sut-1}{article}{
      author={Gabai, David},
       title={Foliations and the topology of {$3$}-manifolds},
        date={1983},
        ISSN={0022-040X},
     journal={J. Differential Geom.},
      volume={18},
      number={3},
       pages={445\ndash 503},
         url={http://projecteuclid.org/euclid.jdg/1214437784},
      review={\MR{723813 (86a:57009)}},
}

\bib{GN-suture}{article}{
      author={Grigsby, J.~Elisenda},
      author={Ni, Yi},
       title={Sutured {K}hovanov homology distinguishes braids from other
  tangles},
        date={2014},
        ISSN={1073-2780},
     journal={Math. Res. Lett.},
      volume={21},
      number={6},
       pages={1263\ndash 1275},
         url={https://doi.org/10.4310/MRL.2014.v21.n6.a4},
      review={\MR{3335847}},
}

\bib{GW-annular}{article}{
      author={Grigsby, J.~Elisenda},
      author={Wehrli, Stephan~M.},
       title={Khovanov homology, sutured {F}loer homology and annular links},
        date={2010},
        ISSN={1472-2747},
     journal={Algebr. Geom. Topol.},
      volume={10},
      number={4},
       pages={2009\ndash 2039},
         url={https://doi.org/10.2140/agt.2010.10.2009},
      review={\MR{2728482}},
}

\bib{GW-color}{article}{
      author={Grigsby, J.~Elisenda},
      author={Wehrli, Stephan~M.},
       title={On the colored {J}ones polynomial, sutured {F}loer homology, and
  knot {F}loer homology},
        date={2010},
        ISSN={0001-8708},
     journal={Adv. Math.},
      volume={223},
      number={6},
       pages={2114\ndash 2165},
         url={https://doi.org/10.1016/j.aim.2009.11.002},
      review={\MR{2601010}},
}

\bib{HN-module}{article}{
      author={Hedden, Matthew},
      author={Ni, Yi},
       title={Khovanov module and the detection of unlinks},
        date={2013},
        ISSN={1465-3060},
     journal={Geom. Topol.},
      volume={17},
      number={5},
       pages={3027\ndash 3076},
         url={https://doi.org/10.2140/gt.2013.17.3027},
      review={\MR{3190305}},
}

\bib{Kh-Jones}{article}{
      author={Khovanov, Mikhail},
       title={A categorification of the {J}ones polynomial},
        date={2000},
        ISSN={0012-7094},
     journal={Duke Math. J.},
      volume={101},
      number={3},
       pages={359\ndash 426},
         url={https://doi.org/10.1215/S0012-7094-00-10131-7},
      review={\MR{1740682}},
}

\bib{Kh-color}{article}{
      author={Khovanov, Mikhail},
       title={Categorifications of the colored {J}ones polynomial},
        date={2005},
        ISSN={0218-2165},
     journal={J. Knot Theory Ramifications},
      volume={14},
      number={1},
       pages={111\ndash 130},
         url={https://doi.org/10.1142/S0218216505003750},
      review={\MR{2124557}},
}

\bib{KM:Alexander}{article}{
      author={Kronheimer, Peter},
      author={Mrowka, Tom},
       title={Instanton {F}loer homology and the {A}lexander polynomial},
        date={2010},
        ISSN={1472-2747},
     journal={Algebr. Geom. Topol.},
      volume={10},
      number={3},
       pages={1715\ndash 1738},
         url={https://doi.org/10.2140/agt.2010.10.1715},
      review={\MR{2683750}},
}

\bib{KM:suture}{article}{
      author={Kronheimer, Peter},
      author={Mrowka, Tomasz},
       title={Knots, sutures, and excision},
        date={2010},
     journal={Journal of Differential Geometry},
      volume={84},
      number={2},
       pages={301\ndash 364},
}

\bib{KM:Kh-unknot}{article}{
      author={Kronheimer, Peter},
      author={Mrowka, Tomasz},
       title={Khovanov homology is an unknot-detector},
        date={2011},
        ISSN={0073-8301},
     journal={Publ. Math. Inst. Hautes \'Etudes Sci.},
      number={113},
       pages={97\ndash 208},
         url={http://dx.doi.org/10.1007/s10240-010-0030-y},
      review={\MR{2805599}},
}

\bib{KM:YAFT}{article}{
      author={Kronheimer, Peter},
      author={Mrowka, Tomasz},
       title={Knot homology groups from instantons},
        date={2011},
        ISSN={1753-8416},
     journal={J. Topol.},
      volume={4},
      number={4},
       pages={835\ndash 918},
         url={http://dx.doi.org/10.1112/jtopol/jtr024},
      review={\MR{2860345 (2012m:57060)}},
}

\bib{KM-Ras}{article}{
      author={Kronheimer, P.~B.},
      author={Mrowka, T.~S.},
       title={Gauge theory and {R}asmussen's invariant},
        date={2013},
        ISSN={1753-8416},
     journal={J. Topol.},
      volume={6},
      number={3},
       pages={659\ndash 674},
         url={https://doi.org/10.1112/jtopol/jtt008},
      review={\MR{3100886}},
}

\bib{Kr-ob}{article}{
      author={Kronheimer, Peter~B},
       title={An obstruction to removing intersection points in immersed
  surfaces},
        date={1997},
     journal={Topology},
      volume={36},
      number={4},
       pages={931\ndash 962},
}

\bib{OS-ss}{article}{
      author={Ozsv\'ath, Peter},
      author={Szab\'o, Zolt\'an},
       title={On the {H}eegaard {F}loer homology of branched double-covers},
        date={2005},
        ISSN={0001-8708},
     journal={Adv. Math.},
      volume={194},
      number={1},
       pages={1\ndash 33},
         url={https://doi.org/10.1016/j.aim.2004.05.008},
      review={\MR{2141852}},
}

\bib{Rob}{article}{
      author={Roberts, Lawrence~P.},
       title={On knot {F}loer homology in double branched covers},
        date={2013},
        ISSN={1465-3060},
     journal={Geom. Topol.},
      volume={17},
      number={1},
       pages={413\ndash 467},
         url={https://doi.org/10.2140/gt.2013.17.413},
      review={\MR{3035332}},
}

\bib{Rolfsen}{book}{
      author={Rolfsen, Dale},
       title={Knots and links},
   publisher={Publish or Perish, Inc., Berkeley, Calif.},
        date={1976},
        note={Mathematics Lecture Series, No. 7},
      review={\MR{0515288}},
}

\bib{Sca}{article}{
      author={Scaduto, Christopher~W.},
       title={Instantons and odd {K}hovanov homology},
        date={2015},
        ISSN={1753-8416},
     journal={J. Topol.},
      volume={8},
      number={3},
       pages={744\ndash 810},
         url={https://doi.org/10.1112/jtopol/jtv012},
      review={\MR{3394316}},
}

\bib{Street}{book}{
      author={Street, Ethan~J.},
       title={Towards an {I}nstanton {F}loer {H}omology for {T}angles},
   publisher={ProQuest LLC, Ann Arbor, MI},
        date={2012},
        ISBN={978-1267-45013-5},
  url={http://gateway.proquest.com/openurl?url_ver=Z39.88-2004&rft_val_fmt=info:ofi/fmt:kev:mtx:dissertation&res_dat=xri:pqm&rft_dat=xri:pqdiss:3514248},
        note={Thesis (Ph.D.)--Harvard University},
      review={\MR{3054928}},
}

\end{biblist}
\end{bibdiv}
\bibliographystyle{hplain}

\end{document}